\documentclass[final,numbook,envcountsame]{svjour3}

\usepackage[colorlinks=true,bookmarks=false,linkcolor=blue,urlcolor=blue,citecolor=blue,breaklinks=true]{hyperref}

\usepackage{breakcites} 

\usepackage{doi}
\usepackage{graphicx}
\usepackage{amsmath} 
\usepackage{amssymb} 
\usepackage{amsfonts} 

\usepackage{amsthm} 
\usepackage[shortlabels]{enumitem} 
\usepackage{verbatim} 
\usepackage{bbm}
\usepackage{mathtools}
\usepackage{tikz-cd}
\usepackage{algorithm}
\usepackage{algorithmicx}
\usepackage{algpseudocode}

\usepackage{ifthen}
\newenvironment{tbd}%
        {%
            \ifthenelse{\isundefined{\showtbds}}%
                    {\comment}%
                    {\begingroup\color{red}[}%
                    }%
         {%
            \ifthenelse{\isundefined{\showtbds}}%
                    {\endcomment}%
                    {]\endgroup}%
          }
          
\newenvironment{journal}%
        {%
            \ifthenelse{\isundefined{\showjournal}}%
                    {\comment}%
                    {\begingroup\color{blue}[}%
                    }%
         {%
            \ifthenelse{\isundefined{\showjournal}}%
                    {\endcomment}%
                    {]\endgroup}%
          }
          
\usepackage{multicol}

\tikzset{
    rot/.style={anchor=south, rotate=90, inner sep=.5mm}
} 

\newtheorem{assumption}[theorem]{Assumption}

\spnewtheorem*{remark*}{Remark}{\itshape}{\rmfamily}
\newcommand{\rank}{\mathrm{rank}}

\newcommand{\Tr}{\mathrm{Tr}}

\newcommand{\id}{\mathrm{id}}
\newcommand{\diag}{\mathrm{diag}}

\newcommand{\mc}[1]{\mathcal{#1}}
\newcommand{\mbb}[1]{\mathbb{#1}}

\newcommand{\TODO}[1]{{\color{red}{[#1]}}}

\newcommand{\argmin}[1]{\underset{#1}{\operatorname{argmin}}}

\newcommand{\transpose}{^\top\! }
\newcommand{\inner}[2]{\left\langle{#1},{#2}\right\rangle}
\newcommand{\innersmall}[2]{\langle{#1},{#2}\rangle}

\newcommand{\Proj}{\mathrm{Proj}}
\newcommand{\Retr}{\mathrm{R}}
\newcommand{\T}{\mathrm{T}}

\newcommand{\Gr}{\mathrm{Gr}}
\newcommand{\St}{\mathrm{St}}

\newcommand{\Or}{{\mathrm{O}(r)}}

\newcommand{\Rnn}{{\mathbb{R}^{n\times n}}}

\newcommand{\Rmn}{{\mathbb{R}^{m\times n}}}

\newcommand{\Rnr}{\mathbb{R}^{n\times r}}

\newcommand{\Rmm}{{\mathbb{R}^{m\times m}}}
\newcommand{\Rmr}{\mathbb{R}^{m\times r}}

\newcommand{\Rn}{{\mathbb{R}^n}}
\newcommand{\Rm}{{\mathbb{R}^m}}

\newcommand{\D}{\mathrm{D}}

\newcommand{\calB}{\mathcal{B}}

\newcommand{\calE}{\mathcal{E}}

\newcommand{\calK}{\mathcal{K}}
\newcommand{\calL}{\mathcal{L}}

\newcommand{\calM}{\mathcal{M}}

\newcommand{\calS}{\mathcal{S}}
\newcommand{\calX}{\mathcal{X}}

\newcommand{\dist}{\mathrm{dist}}

\newcommand{\conv}{\mathrm{conv}}

\newcommand{\reals}{\mathbb{R}} 

\newcommand{\im}{\operatorname{im}}

\newcommand{\Rmnlr}{\reals^{m \times n}_{\leq r}}
\newcommand{\Rmnls}{\reals^{m \times n}_{\leq s}}
\newcommand{\Rttlt}{\reals^{3 \times 3}_{\leq 2}}

\newcommand{\gmin}{g_{\mathrm{min}}}

\newcommand{\deltamax}{\delta_{\mathrm{max}}}
\newcommand{\gammamin}{\gamma_{\mathrm{min}}}
\newcommand{\kappamin}{\kappa_{\mathrm{min}}}
\newcommand{\floor}[1]{\left\lfloor #1 \right\rfloor}
\newcommand{\ceil}[1]{\left\lceil #1 \right\rceil}

\newcommand{\lambdamin}{\lambda_{\mathrm{min}}}

\newcommand{\RR}{\mathbb{R}}

\renewcommand{\tt}[1]{\mathrm{#1}}

\newcommand{\DPGD}{{P\textsuperscript{2}GD}} 

\newcommand{\flatten}{\operatorname{flatten}}

\title{Finding stationary points on bounded-rank matrices:\\ A geometric hurdle and a smooth remedy}
\author{Eitan Levin \and Joe Kileel \and Nicolas Boumal}
\institute{
E. Levin \at Department of Computing and Mathematical Sciences, California Institute of Technology\\ \email{eitanl@caltech.edu}
\and 
J. Kileel \at Department of Mathematics and Oden Institute for Computational Engineering and Sciences, University of Texas at Austin\\ \email{jkileel@math.utexas.edu}
\and
N. Boumal \at Institute of Mathematics, \'Ecole polytechnique f\'ed\'erale de Lausanne (EPFL)\\ \email{nicolas.boumal@epfl.ch}
}
\titlerunning{Finding stationary points on bounded-rank matrices}
\date{\today}

\begin{document}

\maketitle

\begin{abstract}
%
We consider the problem of provably finding
a stationary point of a smooth function to be minimized on
the variety of bounded-rank matrices.
%
%
This turns out to be unexpectedly delicate.
We trace the difficulty back to a geometric obstacle:
On a nonsmooth set, there may be sequences of points along which standard measures of 
stationarity tend to zero, but whose limit points are not stationary. 
We name such events \emph{apocalypses}, as they can cause optimization algorithms to converge to non-stationary points.
We illustrate this explicitly for an existing optimization algorithm on bounded-rank matrices. 
To provably find stationary points, we modify a trust-region method on a standard smooth parameterization of the variety.
The method relies on the known fact that second-order stationary points on the parameter space map to stationary points on the variety.
Our geometric observations and proposed algorithm generalize beyond bounded-rank matrices.
We give a geometric characterization of apocalypses on general constraint sets, 
which implies that Clarke-regular sets do not admit apocalypses. 
Such sets include smooth manifolds, manifolds with boundaries, and convex sets.
Our trust-region method supports parameterization by any complete Riemannian manifold. 

\end{abstract}



\section{Introduction}
\label{sec:intro}

Given a smooth cost function $f \colon \Rmn \to \reals$ with bounded sublevel sets, consider the optimization problem
\begin{align}
	\min_{X \in \Rmn} f(X) \ \textrm{ subject to } \ \rank(X) \leq r.
	\tag{R}
	\label{eq:minfrank}
\end{align}
Its search space is the algebraic variety of matrices of size $m \times n$ and rank at most \nolinebreak $r$:
\begin{align*}
    \Rmnlr = \{ X \in \Rmn : \rank(X) \leq r \}.
\end{align*}
Finding a global minimizer may be hard~\cite{gillis2011low}.
We aim for a 
more realistic goal:
\begin{quote}
	\emph{Design an algorithm which, given an arbitrary $X_0 \in \Rmnlr$, generates a sequence $X_0, X_1, X_2, \ldots$ in $\Rmnlr$ with at least one accumulation point that is (first-order) stationary for~\eqref{eq:minfrank}.}
\end{quote}
As we shall review in Section~\ref{sec:downstairs}, a matrix $X \in \Rmnlr$ is stationary for~\eqref{eq:minfrank} if the directional derivative of $f$ at $X$ along each direction in the tangent cone $\T_X\Rmnlr$ of $\Rmnlr$ at $X$ (given by~\eqref{eq:TXRmnlr} below) is nonnegative. 
Endow $\Rmn$ with the inner product $\inner{A}{B} = \Tr(A^\top\! B)$ and the associated (Frobenius) norm $\|A\| = \sqrt{\inner{A}{A}}$.
Then, $f$ has a gradient $\nabla f$ and 
$X \in \Rmnlr$ is stationary if (and only if)
\begin{align*}
    \|\Proj_{\T_X \Rmnlr}\big(-\nabla f(X)\big)\| = 0,
\end{align*}
where $\Proj_{\T_X \Rmnlr}$ denotes orthogonal projection\footnote{The projection exists but may not be unique. Still, our notation is justified as all projections have the same norm: see Appendix~\ref{apdx:cones}. More generally, we let projectors select one projection arbitrarily.} to the tangent cone at $X$ (given by~\eqref{eq:proj_Tx}). 
In contrast to the smooth manifold of matrices of rank equal to $r$, the variety $\Rmnlr$ is not smooth, which complicates the above goal.

The stated goal turns out to be unexpectedly difficult to achieve. 
There is a rich literature focused on particular cost functions $f$ (see ``prior work'' below).
For the general cost functions we target here, there are a few natural candidate algorithms. 
Of these, the only algorithm of which we are aware that is guaranteed to accumulate at stationary points exploits minute details of the geometry of bounded-rank matrices. It was proposed in response to an open question we posed in the first version of this paper~\cite[v1]{levin2021finding} by Olikier et al.~\cite{olikier2022apocalypse}.
We discuss a few of these algorithms now, and defer the rest to ``prior work''.

Projected gradient descent (PGD), also called iterative hard thresholding~\cite[Thm.~3.4]{jain2014IHT}, would iterate 
\begin{align}
    X_{k+1} & = \Proj_{\Rmnlr}\!\left( X_k - \alpha_k \nabla f(X_k) \right)
    \tag{PGD}
    \label{eq:PGD}
\end{align}
with some step sizes $(\alpha_k)_{k\geq 0}$, where $\Proj_{\Rmnlr}$ denotes metric projection to $\Rmnlr$. The latter amounts to a truncated singular value decomposition (SVD).
It appears that the best known result for PGD applied to~\eqref{eq:minfrank} is that its accumulation points are \emph{Mordukhovich stationary}~\cite[Thm.~3.4]{jia2021augmented}.
This is a strictly weaker and arguably unsatisfactory notion of stationarity~\cite{hosseini2019coneslowrank,li2019rankstationarity}.
Indeed, a matrix $X$ of rank $s \leq r$ is Mordukhovich stationary if it is stationary for the problem of minimizing $f$ on $\Rmnls$ (rather than $\Rmnlr$) and $\rank(\nabla f(X)) \leq \min(m, n) - r$~\cite[\S4]{hosseini2019coneslowrank}.
For $s < r$, there may still exist descent directions in $\T_X\Rmnlr$.

Schneider and Uschmajew~\cite{schneider2015convergence} propose a variant of PGD we refer to as {\DPGD} because it involves two projections.
Their method iterates the following:
\begin{align}
	X_{k+1} & = \Proj_{\Rmnlr}\!\left( X_k + \alpha_k \Proj_{\T_{X_k}\Rmnlr}(-\nabla f(X_k)) \right),
    \tag{{\DPGD}}
	\label{eq:PPGD}
\end{align}
with step sizes $(\alpha_k)_{k\geq0}$ determined through a line-search procedure.
This algorithm has a number of desirable qualities detailed in~\cite{schneider2015convergence}. 
It is arguably more natural than PGD in the sense that its iterates depend on $f$ only through the restriction $f|_{\Rmnlr}$ (see Appendix~\ref{apdx:cones}).
We can interpret it as an extension of Riemannian gradient descent from the smooth manifold of rank-$r$ matrices~\cite{absil_book,vandereycken2013low} to its closure, namely, $\Rmnlr$.
Unfortunately, the convergence of {\DPGD} to a stationary point is proved in~\cite{schneider2015convergence} only under the assumption that the limit point has rank \nolinebreak $r$. 

One may wonder if PGD and {\DPGD} do, in fact, converge to stationary points despite the gap in their theory.
In this paper, we show through an explicit example that this is not the case for {\DPGD}, and we trace back the fundamental obstacle to the delicate geometry of $\Rmnlr$.
This \emph{geometric hurdle} forces a certain level of sophistication in algorithms to solve~\eqref{eq:P} and their analysis.

Acting on this observation, Olikier et al.\ modify {\DPGD} to achieve the stated goal by taking the finer structure of $\Rmnlr$ into account. Their algorithm projects $X_{k+1}$ obtained from~\eqref{eq:PPGD} onto $\RR^{m\times n}_{\leq i}$ for each $i$ between the $\Delta$-rank\footnote{The $\Delta$-rank of a matrix is the number of its singular values that are at least $\Delta$.} of $X_{k+1}$ and $\rank(X_{k+1})$, for a user-chosen $\Delta>0$. The next iterate in their algorithm is chosen to be the projection that has the lowest cost. 

Alternatively, we propose to achieve the stated goal by optimizing over a parametrization of $\Rmnlr$ by a smooth manifold. A common example is the map $(L,R)\mapsto LR^\top$ with $(L,R) \in \Rmr \times \Rnr$. 
This is what we call the \emph{smooth remedy}.
Our proposed algorithm and its analysis generalize beyond bounded-rank matrices, but require second-order information about $f$. For the particular case of bounded-rank optimization, we prove the following result in Section~\ref{sec:examples}:
\begin{theorem}\label{thm:intro_informal_alg}
	There exists an algorithm with the following property.
	Given any three times continuously differentiable $f \colon \Rmn \to \reals$ and any $X_0 \in \Rmnlr$ such that $\{ X \in \Rmnlr : f(X) \leq f(X_0) \}$ is bounded,
	the algorithm generates a sequence $(X_k) \subset \Rmnlr$ whose accumulation points (of which there exists at least one) are first-order stationary for~\eqref{eq:minfrank}.
	Moreover, $f$ is non-increasing along the sequence.
\end{theorem}

Much of our technical considerations extend beyond the bounded-rank variety.
Therefore, wherever meaningful, we discuss the more general problem
\begin{align}
    \min_{x \in \calX} f(x),
    \tag{P}
    \label{eq:P}
\end{align}
where $\calX$ is a subset of a Euclidean space $\calE$ equipped with the inner product $\inner{\cdot}{\cdot}$ and associated norm $\|\cdot\|$.
We need $f$ to be defined merely in a neighborhood of $\calX$ in $\calE$.
Our results depend on $f$ only through its values on $\calX$.
Thus, we can think of $f$ as a (suitably differentiable) extension of $f|_{\calX}$ to a neighborhood of $\calX$, without fear that the somewhat arbitrary choice of extension affects algorithms or conclusions.

\subsection*{The geometric hurdle}

In Section~\ref{subsec:apocalypse_on_mats} we exhibit a (polynomial) cost function $f$ and an initial $X_0$ for which {\DPGD} produces a sequence with the following unfortunate properties:
\begin{enumerate}
	\item All matrices in the sequence have rank $r$,
	\item the sequence converges to a matrix $X$ of rank (strictly) less than $r$,
	\item the stationarity measure $\|\Proj_{\T_{X_k}\Rmnlr}(-\nabla f(X_k))\|$ converges to zero, 
	\item yet $X$ is not stationary.
\end{enumerate}
We call such events \emph{apocalypses}, and we call the points to which such sequences converge \emph{apocalyptic}.
They are an obstacle to optimization algorithm design because it is not immediately clear how to recognize and correct for such bad sequences without fail.
Indeed, along the sequence, all seems well.


As the example suggests, issues can arise when algorithms generate sequences with accumulation points of rank strictly less than $r$, which is where the nonsmooth nature of $\Rmnlr$ is revealed.
This can happen even if the minima of $f$ have rank $r$.

%
%

%
%

Intuitively (though imprecisely), apocalypses arise because limits of tangent cones along convergent sequences may fail to contain the tangent cone at their limit. Hence there may be descent directions in the tangent cone at the limit which are `missed' by the tangent cones along the sequence.
In Section~\ref{sec:apocs_character} we characterize the geometric properties of sets in general and of $\Rmnlr$ in particular that lead to the existence of apocalypses.

We emphasize that apocalyptic points are not saddles; in fact they are not even stationary on the constraint set.
While saddle points can arise on any constraint set (including smooth manifolds), apocalypses can only arise on certain nonsmooth sets.
In particular, we show that Clarke regular sets---a well-known class which includes convex sets and embedded submanifolds with or without boundary~\cite{hesse_noncvx_regularity,rockafellar2009variational}---do not have apocalypses.

\subsection*{A smooth remedy}

In Section~\ref{sec:optimization} we develop an algorithm which meets our stated goal, at the expense of requiring access to second-order information about $f$.
Our approach is to \emph{lift}~\eqref{eq:minfrank} to a smooth space, through a parameterization of $\Rmnlr$.
We let $\varphi \colon \calM \to \Rmn$ be a smooth map on a manifold $\calM$ such that $\varphi(\calM) = \Rmnlr$.
Such a $\varphi$ is a \emph{smooth lift} of $\Rmnlr$.
We provide three examples (the last one is the desingularization lift proposed by Khrulkov and Oseledets~\cite{desingular}):
\begin{align}
	\varphi(L, R) & = LR^\top, & \calM & = \Rmr \times \Rnr, \nonumber \\
	\varphi(U, W) & = UW^\top, & \calM & = \St(m, r) \times \Rnr, \label{eq:low_rk_lifts} \\
	\varphi(X, K) & = X,       & \calM & = \{ (X, K) \in \Rmn \times \Gr(n, n-r) : K \subseteq \ker(X) \}. \nonumber
\end{align}
Above, $\St(m, r) = \{ U \in \Rmr : U^\top U = I_r \}$ is the Stiefel manifold of orthonormal matrices, and $\Gr(n, n-r)$ is the Grassmann manifold of subspaces of dimension $n-r$ in $\Rn$~\cite{absil_book}.
Using any such lift, we rewrite~\eqref{eq:minfrank} equivalently as:
\begin{align}
	\min_{y \in \calM} g(y) \ \textrm{ with } \  g = f \circ \varphi \colon \calM \to \reals.
	\tag{Q}
	\label{eq:Q}
\end{align}
This is an optimization problem on a smooth manifold---see~\cite{absil_book,optimOnMans} for book-length introductions to that topic.
In particular, problem~\eqref{eq:Q} has no apocalypses.
More generally, problem~\eqref{eq:Q} is a lift of~\eqref{eq:P} if $\varphi \colon \calM \to \calE$ is smooth and $\varphi(\calM) = \calX$.

We review in Section~\ref{sec:optimization} what it means for $y \in \calM$ to be second-order stationary for~\eqref{eq:Q}.
For the lift $(L, R) \mapsto LR^\top$, Ha, Liu and Barber~\cite[\S2]{ha2020criticallowrank} show that if $(L, R)$ is second-order stationary for~\eqref{eq:Q} then $LR^\top$ is first-order stationary for~\eqref{eq:minfrank}.
Using similar arguments, we show in~\cite[\S2.3]{eitan_thesis} that the same is true for all lifts in~\eqref{eq:low_rk_lifts}.

Acting on this observation, we construct in Section~\ref{sec:optimization} a modified Riemannian trust-region method for~\eqref{eq:Q} whose accumulation points are second-order stationary under mild conditions on $f$.
It immediately follows that sequences $(y_k) \subset \calM$ generated by that algorithm map through $\varphi$ to sequences of matrices $(X_k)$ in $\Rmnlr$ whose accumulations points are stationary for~\eqref{eq:minfrank}, proving Theorem~\ref{thm:intro_informal_alg} above.

The main difficulty in proving Theorem~\ref{thm:intro_informal_alg} is that the lifted cost function $g = f \circ \varphi$ may not retain desirable properties of $f$ such as having bounded sublevel sets or Lipschitz continuous derivatives. 
We overcome this difficulty in Section~\ref{sec:rebalancing_maps} by introducing \emph{rebalancing maps}, which replace each iterate $y_k$ with a well-chosen element of the fiber $\varphi^{-1}(\varphi(y_k))$.
This ensures that iterates remain in a compact set. 
Our algorithm and its analysis apply to optimization on complete Riemannian manifolds in general.

On an aesthetic note, we design the algorithm such that $y_{k+1}$ depends on $y_k$ only through $X_k = \varphi(y_k)$.
In other words: while the optimization algorithm technically runs on $\calM$, it does correspond to a genuine algorithm on $\Rmnlr$ in the sense that $X_{k+1}$ is a function of $X_k$ but not of the particular $y_k \in \varphi^{-1}(X_k)$ which the algorithm happens to be considering.

In~\cite{schneider2015convergence}, Schneider and Uschmajew say that, because of what we call apocalypses, ``\emph{it will be typically impossible to prove convergence to a rank-deficient critical point by a method which (in regular points) only `sees' projections of the gradient on tangent spaces.}''
Our algorithm overcomes this limitation by using second-order information about $f$ through a smooth lift.
We quantify this statement in Section~\ref{subsec:approxstationarity} by proving a strengthened version of Ha et al.'s aforementioned result~\cite{ha2020criticallowrank}.
\begin{tbd}
proving
\end{tbd}
\begin{journal}
stating
\end{journal}



\subsection*{Prior work}









It is natural to try to ``fix'' PGD or {\DPGD} to reach the stated goal, and indeed several modifications of these algorithms have been studied in the literature. As mentioned above, the only such modification we are aware of that succeeds in doing so was proposed in~\cite{olikier2022apocalypse}.
We mention a few earlier algorithms and analyses below.

In~\cite{schneider2015convergence,gradient_sampling} the authors propose extensions of Riemannian gradient descent to varieties and (more generally) to stratified spaces, but only guarantee that accumulation points of their algorithm are critical on their own (smooth) stratum.
For $\Rmnlr$, this means that if the rank of an accumulation point is $s<r$, then we can only conclude that the point is stationary on $\RR^{m\times n}_{\leq s}$. 
In~\cite{tan2014riemannian,uschmajew2014line,uschmajew2015greedy} the authors study rank increasing schemes, where a guess for the rank of the solution is iteratively increased, and low-rank stationary points are used to warm-start the algorithm aiming for larger ranks.
While such rank-increasing schemes are popular in practice and were observed to help convergence, no theoretical guarantees for general cost functions are given there either.
In~\cite{ZHOU201672,gao2021riemannian} the authors propose more sophisticated rank-adaptive algorithms that combine both rank-increasing and rank-decreasing steps, but here too cannot guarantee convergence to a stationary point on the entire variety $\Rmnlr$.
In~\cite{barber_concavity_parameter} the authors 
give an algorithm that is guaranteed to converge to a local minimizer when initialized sufficiently close to one, though only when assuming restricted strong convexity and Lipschitz-continuous gradient.
In~\cite{ding2020low} the authors propose an averaged projected gradient descent for $\Rmnlr$, but they only guarantee convergence when the cost function is strongly convex and smooth, and the minimizer of the function on $\Rmnlr$ is the global minimum on all of $\RR^{m\times n}$.
In~\cite[Ch.~3]{eitan_thesis}, we describe our own unsuccessful attempt to develop a rank-adaptive algorithm for $\Rmnlr$---and singular varieties more generally---that would be guaranteed to converge to a stationary point.
We were only able to prove convergence to a stationary point on sets with finitely-many singular points, which does not apply to $\Rmnlr$.
In~\cite{pmlr-v130-bi21a,uschmajew_quadratic_lowrk} the authors analyze restricted isometry-type conditions on the cost function that preclude the existence of spurious critical points or local minima.
These results do not preclude the existence of apocalypses however, since apocalyptic points are not critical.

In~\cite{hou2020fast,hou2021asymptotic}, the authors study convergence of {\DPGD} on $\Rmnlr$ applied to cost functions of the form $f(X)=\frac{1}{2}\|X-X_0\|^2$.
They observe that even if {\DPGD} is initialized at a rank $r$ matrix, it may converge to a non-stationary point with rank strictly less than $r$ which they call a ``spurious critical point''. 
Interestingly, this phenomenon is a special case of apocalypses for the above cost function. While the apocalypses arising for cost functions of the above form can be avoided by using a step size of 1, {\DPGD} cannot avoid the apocalypse we construct in Section~\ref{subsec:apocalypse_on_mats} with any choice of (bounded) step size.

At the heart of apocalypses, there is the observation (detailed in Section~\ref{sec:downstairs}) that the map
\begin{align*}
    X \mapsto \|\Proj_{\T_{X}\Rmnlr}(-\nabla f(X))\|
\end{align*}
may fail to be lower semicontinuous.
As a measure of approximate stationarity, this is problematic:
it is zero if and only if $X$ is stationary, yet it may be arbitrarily close to zero even when there are no stationary points nearby.
This fact is closely related to the discontinuity of the tangent cones of $\Rmnlr$: see~\cite{cont_of_tangent_cones} for an in-depth study.
At the very least, this shows that the above measure must be treated with circumspection, especially as part of a stopping criterion for optimization.
It is natural to ask whether other measures of approximate stationarity may fare better.
Unfortunately, we did not identify promising alternatives.

The smooth lift approach to optimization under rank constraints is by and large the most popular.
Some examples of papers using lifts~\eqref{eq:low_rk_lifts} include~\cite{matrix_factorizations_in_optim,mishra2014fixed,globalOpt_matComp,desingular}, to name just a few.
In~\cite[Ch.~2]{eitan_thesis}
we show that first-order stationary points on the lifts in~\eqref{eq:low_rk_lifts} do not necessarily map to first-order stationary points on $\Rmnlr$, although second-order stationary points do. We also show that there may be local minima on these lifts that do not map to local minima on $\Rmnlr$. For the $LR^\top$ lift, this issue is resolved with balancing (see Example~\ref{ex:LR}).

\section{A geometric hurdle: apocalypses}
\label{sec:downstairs}

In this section, we ask what can go wrong when we attempt to solve the problem~\eqref{eq:P} directly by running a gradient-based algorithm on the possibly nonsmooth set $\calX$.
Since finding the global minimum of $f$ on $\calX$ may be too ambitious, our more modest goal is to find a first-order stationary point of $f$ on $\calX$, as defined below after a couple of preliminary definitions.
\begin{definition} \label{def:cones}
	A set $K \subseteq \calE$ is a \emph{cone} if $v \in K \implies \alpha v \in K$ for all $\alpha > 0$.
\end{definition}
\begin{definition} \label{def:tangentcone}
	The \emph{tangent cone} to $\calX$ at $x \in \calX$ is the set
	\begin{align*}
		\T_x\calX & = \left\{ v = \lim_{k \to \infty} \frac{x_k - x}{\tau_k} : x_k \in \calX, \tau_k > 0 \textrm{ for all } k, \tau_k \to 0 \textrm{ and the limit exists} \right\}.
	\end{align*}
    This is a closed (but not necessarily convex) cone~\cite[Lem.~3.12]{nonlin_optim}.
\end{definition}
The cone $\T_x\calX$ is also called the \emph{contingent} or \emph{Bouligand} tangent cone \cite[\S2.7]{clarke2008nonsmooth}. If $\calX$ is a smooth manifold, Definition~\ref{def:tangentcone} coincides with the usual tangent space to a manifold~\cite[Ex.~6.8]{rockafellar2009variational}.
\begin{definition} \label{def:criticalP}
	A point $x \in \calX$ is (first-order) \emph{stationary} for~\eqref{eq:P} if $\D f(x)[v] \geq 0$ for all $v \in \T_x\calX$, where $\D f(x)[v] = \lim_{t \to 0} \frac{f(x+tv) - f(x)}{t}$ denotes directional derivatives.
\end{definition}
In words, $x$ is stationary if the cost function is non-decreasing to first order along all tangent directions at $x$.
Local minima of~\eqref{eq:P} are stationary~\cite[Thm.~3.24]{nonlin_optim}. 
The notion of stationarity only depends on the values of $f$ on $\calX$, see Lemma~\ref{lem:nablabarfinvariance}.

With a Euclidean structure on $\calE$, we can redefine stationarity in several equivalent ways.
Let $\nabla f(x)$ denote the Euclidean gradient of $f$ at $x$,
that is, $\nabla f(x)$ is the unique element of $\calE$ such that $\inner{\nabla f(x)}{v} = \D f(x)[v]$ for all $v \in \calE$.
Let us also define polars.
\begin{definition} \label{def:polardual}
	The \emph{polar} of a cone $K \subseteq \calE$ is the closed, convex cone defined by $K^\circ = \{ w \in \calE : \inner{w}{v} \leq 0  \textrm{ for all } v \in K \}$.
\end{definition}
See Lemma~\ref{lem:coneproperties} for basic properties of polar cones. We then have the following equivalent characterizations of stationarity:
\begin{proposition}\label{prop:equiv_defns_of_1crit}
With $f \colon \calE \to \reals$ differentiable, these are equivalent for $x \in \calX$:
\begin{enumerate}[(a)]
    \item $x$ is stationary for $f$ on $\calX$.
    \item $-\nabla f(x)\in (\T_x\calX)^{\circ}$.
    \item $\|\Proj_{\T_x\calX}(-\nabla f(x))\| = 0$. 
\end{enumerate}
\end{proposition}
\begin{proof}
The equivalence of (a) and (b) follows from the definitions of stationarity and of the polar. The equivalence of (b) and (c) follows from Proposition~\ref{prop:normprojtocone}.
\end{proof}

Given the equivalent characterization of stationarity in Proposition~\ref{prop:equiv_defns_of_1crit}(c), it is natural to consider the following relaxed notion of stationarity. 
\begin{definition} \label{def:criticalPapprox}
    A point $x \in \calX$ such that $\|\Proj_{\T_x\calX}(-\nabla f(x))\| \leq \epsilon$ for some $\epsilon \geq 0$ is called \emph{$\epsilon$-stationary} for~\eqref{eq:P}.
\end{definition}
We can get an alternative notion of approximate stationarity by relaxing Proposition 2.5(b) instead, but this turns out to be less natural, see Theorem~\ref{thm:invarianceprojsummary}. 
\begin{journal}
(We can get an alternative notion of approximate stationarity by relaxing Proposition 2.5(b) instead, but this turns out to be less natural~\cite[Thm.~A.9]{levin2021finding}.)
\end{journal}
\begin{tbd}
We can get an alternative notion of approximate stationarity by relaxing Proposition 2.5(b) instead, but this turns out to be less natural, see Theorem~\ref{thm:invarianceprojsummary}. 
\end{tbd}

As we show in Appendix~\ref{apdx:cones}, the (non-empty) set $\Proj_{\T_x\calX}(-\nabla f(x))$ depends on $f$ only through $f|_\calX$, and all of its elements have the same norm.
This norm measures the maximal linear rate of decrease of $f$ one can obtain by moving away from $x$ along a direction in the tangent cone at $x$.
Some gradient-based algorithms use this measure of approximate stationarity. 
For example, the {\DPGD} method~\eqref{eq:PPGD} proposed in~\cite{schneider2015convergence} computes a projection $\Proj_{\T_x\calX}(-\nabla f(x))$ at iterate $x$ and terminates if its norm is sufficiently small.
Otherwise, the projection is used to generate the next iterate by moving along a curve on $\calX$ with initial velocity matching the projection. 
However, as we proceed to show below, this notion of approximate stationarity should be treated with circumspection.

Feasible optimization algorithms for~\eqref{eq:P} generate sequences of points $(x_k)_{k\geq 1}$ on $\calX$.
Possibly passing to a subsequence, assume $(x_k)$ has a limit $x \in \calX$.
When $\calX$ is a smooth manifold, it is clear that $x$ is stationary if and only if the points $x_k$ are $\epsilon_k$-stationary with $\epsilon_k \to 0$.
However, when $\calX$ is nonsmooth both the ``if'' and ``only if'' parts can fail.
The failure of the ``if'' part is most significant for optimization purposes (see Definition~\ref{def:serendipities} and surrounding comments below for the ``only if'' part).
It implies that there may exist a sequence along which the stationarity measure in Definition~\ref{def:criticalPapprox} goes to zero, yet whose limit is not stationary.
We call them \emph{apocalypses}:
\begin{definition}\label{def:apocalypses}
    A point $x\in\calX$ is called \emph{apocalyptic} if there exists a sequence $(x_k)_{k\geq 1}\subseteq\cal X$ converging to $x$ and a continuously differentiable function $f$ such that the $x_k$ are $\epsilon_k$-stationary for~\eqref{eq:P} with some sequence of $\epsilon_k > 0$ converging to 0, but $x$ is not stationary for~\eqref{eq:P}. 
 	Such a triplet $(x, (x_k)_{k\geq 1}, f)$ is called an \emph{apocalypse}.
\end{definition}
Note that if $x^*$ is apocalyptic then the measure of approximate stationarity $x\mapsto \|\Proj_{\T_x\calX}(-\nabla f(x))\|$ is not lower semicontinuous at $x^*$ since there exists $x_k\to x^*$ such that $\lim_k\|\Proj_{\T_{x_k}\calX}(-\nabla f(x_k))\|=0 < \|\Proj_{\T_{x^*}\calX}(-\nabla f(x^*))\|$. 

The existence of apocalypses is a clear hazard for optimization algorithms operating on~\eqref{eq:P}.
Indeed, gradient-based algorithms would happily follow a sequence of points that become increasingly close to being stationary.
Such an algorithm would necessarily terminate after a finite amount of time, thus, before reaching the limit.
The last computed iterate, even though it is approximately stationary, may not be close to any stationary point.


We emphasize that existence of apocalyptic points is a purely geometric phenomenon that only depends on the constraint set $\calX$, not on any particular cost function or algorithm.
On the other hand, a particular cost function may not have any apocalypses at all, and a particular algorithm may not follow any apocalypse.
This will become clear in our characterization of apocalypses in Theorem~\ref{thm:apocalypse_char}. 

The complementary notion to apocalypses is a sequence of points that are not approximately stationary yet nevertheless converge to a stationary point.
We call this phenomenon a \emph{serendipity}:
\begin{definition}\label{def:serendipities}
    A point $x \in \calX$ is called \emph{serendipitous} if there exist a sequence $(x_k)_{k\geq 1}\subseteq\calX$ converging to $x$, a continuously differentiable function $f$, and $\epsilon > 0$ such that none of the $x_k$ are $\epsilon$-stationary for~\eqref{eq:P}, yet $x$ is stationary for~\eqref{eq:P}. 
    Such a triplet $(x, (x_k)_{k\geq 1}, f)$ is called a \emph{serendipity}.
\end{definition}
Note that if $x^*$ is serendipitous, then the measure of approximate stationarity $x\mapsto \|\Proj_{\T_x\calX}(-\nabla f(x))\|$ is not upper semicontinuous. 

Of course, serendipities do not exist if
the problem is unconstrained or, more generally, if $\calX$ is a smooth manifold. 
However, it is easy to construct serendipities for other constrained problems.
For example, with $\calX = [0, 1]$ as a subset of $\calE = \reals$ the point $x = 0$ is serendipitous.
(To verify this, let $x_k = 1/k$ and $f(x) = x$.)
More generally, we expect serendipities at the boundary of a search space $\calX$.
The existence of serendipities illustrates that a point that is not approximately stationary can still be arbitrarily close to a stationary point. 
Like apocalypses, we later characterize serendipities purely geometrically in terms of the tangent cones to $\calX$.

\begin{remark}
    Apocalypses and serendipities can only occur at points of $\calX$ where the set is not (locally) a smooth manifold.
    However, it is not necessarily the case that such points are apocalyptic or serendipitous.
    For example, the double parabola in $\reals^2$ defined by $y^2=x^4$ is not smooth near the origin, yet the origin is neither apocalyptic nor serendipitous.
\end{remark}

In the next two sections we focus on the case of bounded-rank optimization, then we circle back to the general case in Section~\ref{sec:apocs_character}.

\subsection{Tangent cones to the bounded-rank variety}


In this section, we recall what the tangent cones to the bounded-rank variety $\Rmnlr$ are, and how to project onto them.
Suppose $X\in\Rmnlr$ has SVD $X=U\Sigma V^\top$ where $U\in \Rmm,\ V\in \Rnn$ are orthogonal and $\Sigma=\diag(\sigma_1,\ldots,\sigma_s,0,\ldots,0)\in\RR^{m\times n}$ is diagonal but possibly non-square.
Here $s=\rank(X)\leq r$.
As stated in~\cite[Thm.~3.2]{schneider2015convergence} we have
\begin{align}
    \T_X\Rmnlr & = \left\{U\begin{bmatrix} M_1 & M_2 \\ M_3 & M_4\end{bmatrix}V^\top: M_1\in\RR^{s\times s},\ M_4\in\RR^{(m-s)\times(n-s)}_{\leq r-s}\right\} \label{eq:TXRmnlr} \\
               & = \{M\in\RR^{m\times n}: \rank(\Proj_{\tt{col}(X)^{\perp}}M\Proj_{\tt{row}(X)^{\perp}}) \leq r-s\} \nonumber,
\end{align}
where $\tt{col}(X)$ and $\tt{row}(X)$ denote the column and row spaces of $X$. 
The projection of $M\in\Rmn$ onto $\T_X\Rmnlr$ is given by~\cite[Alg.~2]{schneider2015convergence}
\begin{equation}\label{eq:proj_Tx}
    \Proj_{\T_X\Rmnlr}(M) = U\begin{bmatrix} M_1 & M_2 \\ M_3 & \mathrm{P}_{\leq r-s}(M_4)\end{bmatrix}V^\top,\quad \textrm{if } M = U\begin{bmatrix} M_1 & M_2 \\ M_3 & M_4\end{bmatrix}V^\top,
\end{equation} 
where we abbreviated $\mathrm{P}_{\leq r-s}(M_4)\coloneqq \Proj_{\RR^{(m-s)\times(n-s)}_{\leq r-s}}(M_4)$, which is obtained by zeroing out all but the top $r-s$ singular values of $M_4$.

The polar of the tangent cone is also available in explicit form.
There are two cases to consider.
Either $\rank(X) = r$, in which case $X$ is a smooth point of the variety: the tangent cone at $X$ is a linear space described as in~\eqref{eq:TXRmnlr} with $M_4 = 0$; its polar is its orthogonal complement.
Otherwise $\rank(X) < r$, in which case $\T_X\Rmnlr$ contains all rank-1 matrices. Since their convex hull is all of $\Rmn$, we have $(\T_X\Rmnlr)^{\circ} = \{0\}$.
To summarize:
\begin{align*}
    (\T_X\Rmnlr)^{\circ} & =
    \begin{cases}
        \{M \in \Rmn : \Proj_{\tt{col}(X)}M = M\Proj_{\tt{row}(X)} = 0\}
        & \textrm{if } \rank(X) = r, \\
        \{ 0 \} & \textrm{if } \rank(X) < r.
    \end{cases}
\end{align*}
Thus, $X$ with $\rank(X) < r$ is stationary for $f$ on $\Rmnlr$ if and only if $\nabla f(X) = 0$. 

\subsection{Apocalypses on $\Rmnlr$ and implications for {\DPGD}} \label{subsec:apocalypse_on_mats}

The subset of $\Rmnlr$ consisting of rank-$r$ matrices is the smooth locus of $\Rmnlr$: none of those points are apocalyptic.
In contrast, all the other matrices are singular points of the variety $\Rmnlr$.
As it turns out, they are all apocalyptic:
\begin{proposition}\label{prop:apocs_on_low_rk_var}
    On the set $\Rmnlr$ with $r < \min(m, n)$, any matrix of rank strictly less than $r$ is apocalyptic.
\end{proposition}
\begin{proof}
        Let $\bar X = U\Sigma V^\top$ be an SVD of $\bar X \in \Rmnlr$ of size $r$, that is, $U \in \Rmr$, $V \in \Rnr$ and $\Sigma = \diag(\sigma_1, \ldots, \sigma_r)$ with $U^\top U = I_r$, $V^\top V = I_r$ and $\sigma_1 \geq \cdots \geq \sigma_r$.
        Assuming $\rank(\bar X) = s < r$, we have $\sigma_r = 0$.
        Since $r < \min(m, n)$, we can select $u \in \Rm$ and $v \in \Rn$ each of unit norm such that $U^\top u = 0$ and $V^\top v = 0$.
        Consider the function $f(X) = u^\top X v$, whose gradient in $\Rmn$ is $\nabla f(X) = uv^\top$.
        Notice that $\bar X$ is not stationary for the problem of minimizing $f$ over $\Rmnlr$ because $\rank(\bar X) < r$ and $\nabla f(\bar X) \neq 0$.
        Now consider the sequence $(X_k)_{k\geq 1}$ on $\Rmnlr$ defined by $X_k = U\Sigma_k V^\top$, where $\Sigma_k = \diag(\sigma_1, \ldots, \sigma_{s}, 1/k, \ldots, 1/k)$.
        It is clear that $X_k \in \Rmnlr$ for all $k$ and that $\lim_{k\to\infty} X_k = \bar X$.
        Moreover, $-\nabla f(X_k) = -uv^\top\in (\T_{X_k}\Rmnlr)^{\circ}$ for all $k$ since
        \begin{align*}
            \Proj_{\tt{col}(X_k)}(-\nabla f(X_k)) & = UU^\top(-uv^\top) = 0, \textrm{ and } \\
            (-\nabla f(X_k))\Proj_{\tt{row}(X_k)} & = (-uv^\top) VV^\top = 0.
        \end{align*}
        Therefore,
        we conclude that the triplet $(\bar X, (X_k)_{k \geq 1}, f)$ is an apocalypse.
\end{proof}

Not only do apocalypses exist on $\Rmnlr$, but also the {\DPGD} algorithm of Schneider and Uschmajew~\cite{schneider2015convergence} can be made to follow an apocalypse. We construct an explicit example of this below.
Thus, in this adversarially designed situation, {\DPGD} converges to a non-stationary point.
The {\DPGD} algorithm applied to the set
of bounded-rank matrices iterates~\eqref{eq:PPGD} 
where the step size
$\alpha_k$ is chosen via backtracking line-search to satisfy the sufficient decrease condition
\begin{align} \label{eq:sufficient_decrease_condition}
	f(X_k)-f(X_{k+1}) & \geq \tau\alpha_k\Big\|\Proj_{\T_{X_k}\Rmnlr}(-\nabla f(X_k))\Big\|^2,
\end{align}
for some $\tau\in(0,1)$ chosen by the user.

We construct a function $f\colon \Rttlt \to \reals$ and a point $X_0$ such that the sequence $(X_k)_{k\geq 0}$ generated by~\eqref{eq:PPGD} starting from $X_0$ converges to a limit $X$ and the triplet $(X, (X_k)_{k\geq 0}, f)$ forms an apocalypse.
In other words, we have $\|\Proj_{\T_{X_k}\Rttlt}(-\nabla f(X_k))\| \to 0$ so we appear to converge to a stationary point, yet $\|\Proj_{\T_X\Rttlt}(-\nabla f(X))\| \neq 0$ so the limit is not stationary.

To this end, we first construct a convex function $Q \colon \reals^{2\times 2}\to \reals$ that would take infinitely many steps to minimize using gradient descent.
Let
\begin{align*}
	Q(Y) & = \frac{1}{2}\| D(Y-Y^*)\|^2, & \textrm{ with } & & D & = \diag(1, 1/2), & Y^* & = \diag(1, 0).
\end{align*}
With a constant step size $\alpha$, gradient descent on $Q$ iterates
\begin{align*}
	Y_{k+1} & = Y_k - \alpha\nabla Q(Y_k) = (I-\alpha D^2)Y_k + \alpha D^2Y^*.
\end{align*}
Applying this formula inductively we get an explicit expression:
\begin{align*}
	Y_k-Y^* & = (I-\alpha D^2)^k(Y_0 - Y^*).
\end{align*}
Initializing from $Y_0 = \diag(2, 1)$, this further simplifies to:
\begin{align*}
	Y_k & = Y^* + \diag((1-\alpha)^k, (1-\alpha/4)^k).
\end{align*}
With appropriate $\alpha$, this sequence converges linearly to the unique global minimizer $Y^*$.
The rate of convergence is optimal with $\alpha = 8/5$, in which case the discrepancy decays as $(3/5)^k$.
Note that convergence still requires infinitely many steps.
Note also that $\rank(Y_k) = 2$ for all $k\geq 0$, while $\rank(Y^*) = 1$.

We now construct our bad example.
Let $f\colon \Rttlt \to \reals$ be given by
\begin{align*}
	f(X) & = Q(X_{1:2,1:2}) - \frac{(X_{3,3}+1)^2}{2} + \frac{X_{3,3}^4}{4},
\end{align*}
where $X_{1:2,1:2}$ is the upper-left $2\times 2$ submatrix of $X$ and $X_{3,3}$ is its bottom-right entry.
Notice that $f$ is a sum of functions of only the upper-left and only the lower-right diagonal blocks of its input, respectively, and is independent of the remaining entries.
In particular, this makes it possible to claim that the global minimum for $f$ is attained at $X^* = \mathrm{blkdiag}(Y^*, x_0) = \diag(1, 0, x_0) \in \Rttlt$, where $\mathrm{blkdiag}$ constructs a block-diagonal matrix and $x_0 \approx 1.32$ is the unique global minimizer of the univariate quartic $x \mapsto -\frac{(x+1)^2}{2}+\frac{x^4}{4}$.
This confirms that $f$ is bounded below.

The tangent cone to $\Rttlt$ at any matrix of the form $X = \mathrm{blkdiag}(Y,0)$ with $Y \in \reals^{2\times 2}$ and $\rank(Y) = 2$ is the linear space given by~\eqref{eq:TXRmnlr}, namely,
\begin{align*}
	\T_X\Rttlt & = \left\{\begin{pmatrix} A & b\\ c^\top & 0\end{pmatrix} : A\in\reals^{2\times 2},\ b,c\in \reals^2 \right\}.
\end{align*}
The orthogonal projector $\Proj_{\T_X\Rttlt}$ onto this subspace simply zeros out the $(3, 3)$ entry (bottom right) by~\eqref{eq:proj_Tx}.
Therefore, for any such point $X$ we have
\begin{align*}
	\nabla f(X) & = \begin{pmatrix} \nabla Q(Y) & 0\\ 0 & -1 \end{pmatrix}, & \Proj_{\T_X\Rttlt}(-\nabla f(X)) & = \begin{pmatrix} -\nabla Q(Y) & 0\\ 0 & 0 \end{pmatrix}.
\end{align*}
We initialize {\DPGD} at $X_0 = \mathrm{blkdiag}(Y_0, 0) = \diag(2, 1, 0)$ and use the constant step size $\alpha = 8/5$ (see below regarding line-search).
Since $X_0$ is of the form above, the next iterate is 
\begin{align*}
	X_1 & = \Proj_{\Rttlt}\!\left( X_0+\alpha\Proj_{\T_{X_0}\Rttlt}(-\nabla f(X_0)) \right) \\
        & = \mathrm{blkdiag}(Y_0-\alpha \nabla Q(Y_0), 0) = \mathrm{blkdiag}(Y_1, 0).
\end{align*}
Suppose that the $k$th iterate has the form $X_k = \mathrm{blkdiag}(Y_k, 0)$.
Using $\rank(Y_k) = 2$, the same argument applies inductively and we have $X_{k+1} = \mathrm{blkdiag}(Y_{k+1}, 0)$.
Thus, we conclude that $X_0, X_1, X_2, \ldots$ converges to $\bar{X} \triangleq \mathrm{blkdiag}(Y^*, 0)$ in infinitely many steps.
Moreover,
\begin{align*}
    \|\Proj_{\T_{X_k}\Rttlt}(-\nabla f(X_k))\| = \|\nabla Q(Y_k)\| = \frac{\sqrt{17}}{4} (3/5)^k \to 0.
\end{align*}
In words: the stationarity measure converges to zero along $X_0, X_1, X_2, \ldots$

However, $\bar{X}$ is not stationary for $f$ on $\Rttlt$.
Indeed, since $\rank(Y^{*}) = 1 < 2$, the tangent cone to $\Rttlt$ at $\bar{X}$ is given by
\begin{align*}
	\T_{\bar{X}}\Rttlt & = \left\{\begin{pmatrix} a & b^\top \\ c & E \end{pmatrix}: a \in \reals, \ b, c \in \reals^2 \textrm{ and } \rank(E) \leq 1\right\}.
\end{align*}
Since $-\nabla f(\bar{X}) = \diag(0, 0, 1) \in \T_{\bar{X}}\Rttlt$, 
we have
\begin{align*}
	\Proj_{\T_{\bar{X}}\Rttlt}(-\nabla f(\bar{X})) = -\nabla f(\bar{X}) \neq 0.
\end{align*}
Thus, the triplet $(\bar{X}, (X_k)_{k\geq 0}, f)$ is an apocalypse on $\Rttlt$, and moreover {\DPGD} follows that apocalypse.

To close the construction of our example, we address the question of step size selection.
Notice that our step size $\alpha$ satisfies the sufficient decrease condition~\eqref{eq:sufficient_decrease_condition} with $\tau = 1/5$.
Indeed, with $M = I-(I-\alpha D)^2$ we have:
\begin{align*}
    f(X_k) - f(X_{k+1}) & = Q(Y_k) - Q(Y_{k+1}) = \frac{1}{2}\inner{MD(Y_k-Y^*)}{D(Y_k-Y^*)} \\
    & \geq \frac{\lambda_{\min}(M)}{2}\| D(Y_k-Y^*)\|^2 = \tau\alpha\|\nabla Q(Y)\|^2 \\
    & = \tau\alpha\|\Proj_{\T_{X_k}\Rttlt}(-\nabla f(X_k))\|^2.
\end{align*}
Thus, if we initialize the Armijo backtracking line-search of {\DPGD} with the proposed value of $\alpha$ and use the above value of $\tau$, that initial step size is accepted every time.
Consequently, {\DPGD} with such line-search indeed follows the proposed apocalypse.


\begin{remark}
    Contemplating the particular apocalypse constructed above, it is easy to devise a ``fix''.
    For example, 
    we can project $X_k$ to the lesser rank matrices when its smallest nonzero singular value drops below a threshold.
    In this particular instance, that would bring us to $\bar{X}$ in finitely many steps.
    From there, the negative gradient points in a good direction and we can proceed.
    Such strategies based on singular value truncation have been considered by a number of authors, e.g.~\cite{olikier2022apocalypse,ZHOU201672,gao2021riemannian} and~\cite[\S3.3]{eitan_thesis}. 
    Of these, only the (arguably) elaborate algorithm of~\cite{olikier2022apocalypse} described in Section~\ref{sec:intro} is guaranteed to converge to a stationary point. 
    We believe apocalypses are the common underlying geometric hurdle that complicates these algorithms and their analyses.
\end{remark}

\subsection{Characterization of apocalypses}\label{sec:apocs_character}
Having constructed an apocalypse on $\Rmnlr$, we turn to characterizing apocalypses on a general subset $\calX$ of a Euclidean space $\calE$.
We seek a description in geometric terms involving limits of tangent cones. 
\begin{definition}[{\protect\cite[Def.~4.1, Ex.~4.2]{rockafellar2009variational}}] \label{def:lims_of_sets}
    Let $(C_k)_{k\geq 1}$ be a sequence of subsets of a Euclidean space $\calE$.
    The outer and inner limits of the $C_k$ are defined as follows:
    \begin{enumerate}
        \item Outer limit:
        \begin{equation} \label{eq:limsup_of_sets}
            \limsup_{k\to\infty} C_k = \left\{ x \in \calE : \liminf_{k\to\infty} \dist(x, C_k) = 0 \right\}.
        \end{equation}
        \item Inner limit:
        \begin{equation} \label{eq:liminf_of_sets}
            \liminf_{k\to\infty} C_k = \left\{ x \in \calE : \limsup_{k\to\infty} \dist(x, C_k) = 0 \right\}.
        \end{equation}
    \end{enumerate}
    Above, $\dist(x, C_k) = \inf_{y \in C_k} \|y - x\|$ denotes distance from $x$ to the set $C_k$.
\end{definition}
The above notions allow for a clear characterization of apocalypses as follows.
We denote by $\overline{\mathrm{conv}}(A)$ the closure of the convex hull of $A$.
\begin{theorem} \label{thm:apocalypse_char}
    A triplet $(x,(x_k)_{k\geq 1},f)$ is an apocalypse on $\calX$ if and only if $-\nabla f(x)\in \left(\limsup_k\T_{x_k}\calX\right)^{\circ}\setminus (\T_x\calX)^{\circ}$. 
    A point $x\in\calX$ is apocalyptic if and only if there exists a sequence $(x_k)_{k\geq 1}\subseteq\mc X$ converging to $x$ such that $\left(\limsup_k\T_{x_k}\calX\right)^{\circ}\not\subseteq (\T_x\calX)^{\circ}$, or equivalently, $\T_x\calX\not\subseteq\overline{\mathrm{conv}}\left(\limsup_k\T_{x_k}\calX\right)$.
\end{theorem}
See Appendix~\ref{apdx:pf_of_apocs_char} for the proof.
Theorem~\ref{thm:apocalypse_char} shows that apocalypses arise when tangent cones along a convergent sequence `miss' directions that only become visible in the limit.
It is instructive to revisit the proof of Proposition~\ref{prop:apocs_on_low_rk_var} from this perspective. 

How pervasive are apocalypses?
We proceed to show that they do not occur on Clarke regular sets.
Recall that $\calX \subseteq \calE$ is \emph{locally closed} in $\calE$ at $x$ if there exists a closed neighborhood $V$ of $x$ such that $\calX \cap V$ is closed in $\calE$~\cite[p28]{rockafellar2009variational}.
\begin{definition}[{\protect\cite[Cor.~6.29(g)]{rockafellar2009variational}}]
    A locally closed set $\calX$ is \emph{Clarke regular} at $x\in\calX$ if $\T_x\calX\subseteq \liminf_k\T_{x_k}\calX$ for all sequences $(x_k)_{k\geq 1}$ in $\calX$ converging to $x$.
\end{definition}
Clarke regularity implies several other notions of regularity used in the literature~\cite{hesse_noncvx_regularity}.
In particular, convex sets and smooth embedded submanifolds are Clarke regular~\cite[Ex.~6.8, Thm.~6.9]{rockafellar2009variational}.
Sets of the form $\calX = \{ x : h(x) \in Y \}$ satisfying a certain constraint qualification are Clarke regular at $x$ if $Y$ is Clarke regular at $h(x)$~\cite[Thm.~6.31]{rockafellar2009variational}.
We also show below that manifolds with boundaries are Clarke regular.
\begin{corollary}\label{cor:apocalypse_suff_cond}
    We have the following implications:
    \begin{align*}
        &\text{$\calX$ is Clarke regular at $x\in\calX$}\\[0.1in]
        &\implies \text{$\T_x\calX\subseteq\limsup_k\T_{x_k}\calX$ for all sequences $(x_k)_{k\geq 1}$ converging to $x$} \\
        &\implies \text{$\T_x\calX\subseteq\overline{\mathrm{conv}}\left(\limsup_k\T_{x_k}\calX\right)$ for all sequences $(x_k)_{k\geq 1}$ converging to $x$}\\
        &\iff \text{$x$ is not apocalyptic}.
    \end{align*}
\end{corollary}
\begin{proof}
    Clarke regularity is equivalent to $\T_x\calX\subseteq \liminf_k\T_{x_k}\calX$ for all sequences $(x_k)_{k\geq 1}$ converging to $x$, and because $\liminf_k\T_{x_k}\calX\subseteq\limsup_k\T_{x_k}\calX$ we get the first implication.
    The second implication is trivial, and the last one is Theorem~\ref{thm:apocalypse_char}.
\end{proof}
We conjecture that none of the implications in Corollary~\ref{cor:apocalypse_suff_cond} can be reversed.

If $\calX$ is locally diffeomorphic to a Clarke regular set, then it is Clarke regular. 
That is because Clarke regularity only depends on the tangent cones to $\calX$, whose relevant properties are preserved under local diffeomorphisms, see Appendix~\ref{apdx:clarke_regularity_loc_diffeo}.
\begin{tbd}, see Appendix~\ref{apdx:clarke_regularity_loc_diffeo}.\end{tbd}
\begin{journal}. We provide the details in~\cite[App.~F]{levin2021finding}.\end{journal}
In particular,
\begin{corollary}
    A smooth embedded submanifold of $\calE$ with a boundary~\cite[p120]{lee_smooth} is Clarke regular, hence has no apocalypses.
\end{corollary}
\begin{proof}
    Every point in a manifold with boundary has a neighborhood diffeomorphic to either a half-space or an entire Euclidean space. Both are convex and in particular, Clarke regular. 
\end{proof}


    A similar argument gives the following characterization of serendipities:
    \begin{theorem}\label{thm:serendipities_char}
        A point $x\in\calX$ is serendipitous if and only if there exists a sequence $(x_k)_{k\geq 1}\subseteq\mc X$ converging to $x$ such that $(\limsup_k\T_{x_k}\calX)^{\circ}\not\supseteq (\T_x\calX)^{\circ}$, or equivalently, $\limsup_k\T_{x_k}\calX \not\subseteq \overline{\tt{conv}}(\T_x\calX)$.
        A sufficient condition for $\calX$ not to have any serendipities is that $\limsup_k\T_{x_k}\calX\subseteq \T_{\lim_kx_k}\calX$ for any convergent sequence $(x_k)_{k\geq 1}\subseteq\calX$.
    \end{theorem}
    This sufficient condition is equivalent to requiring the set-valued map $x\mapsto \T_x\calX$ to be outer semicontinuous in the sense of~\cite[Def.~5.4]{rockafellar2009variational}.
    Comparing the characterization of apocalypses in Theorem~\ref{thm:apocalypse_char} to that of serendipities in Theorem~\ref{thm:serendipities_char}, we see that the two are indeed complementary notions.


\section{Optimization through a smooth lift}
\label{sec:optimization}


In order to minimize a function $f \colon \calX \to \reals$ on a possibly complicated subset $\calX$ of a linear space $\calE$ as in~\eqref{eq:P},
we resort to a smooth map $\varphi \colon \calM \to \calE$ defined on a smooth manifold $\calM$ and such that $\varphi(\calM) = \calX$ (in words: $\varphi$ is a parameterization of $\calX$).
Then, minimizing $f$ on $\calX$ is equivalent to minimizing $g = f \circ \varphi$ on $\calM$, as expressed in problem~\eqref{eq:Q}.
If $f$ is smooth, then $g$ is smooth by composition. 
For our main problem~\eqref{eq:minfrank}, $\calE = \Rmn$, $\calX = \Rmnlr$ and we have listed three possible smooth lifts in the introduction, see eq.~\eqref{eq:low_rk_lifts}.

To find a stationary point of $f$ on $\calX$ (Definition~\ref{def:criticalP}), we seek a point $y$ on $\calM$ such that $\varphi(y)$ is stationary for $f$.
Crucially, it is \emph{not} sufficient to find a first-order stationary\footnote{A point $y \in \calM$ is first-order stationary for $g$ if for all smooth curves $c \colon \reals \to \calM$ passing through $y$ at $t = 0$, it holds that $(g \circ c)'(0) = 0$. If $\calM$ is embedded in a linear space, this is equivalent to Definition~\ref{def:criticalP} by~\cite[Prob.~3-8]{lee_smooth}.} point of $g$.\footnote{For example, it is easy to check that $(0, 0)$ is always stationary for $g$ with the $(L, R) \mapsto LR^\top$ lift, even though $X = LR^\top = 0$ need not be stationary for all $f$.}
However, for the three lifts in~\eqref{eq:low_rk_lifts} it turns out that \emph{second}-order stationary points of $g$ on $\calM$ \emph{do} always map to stationary points of $f$---see~\cite{ha2020criticallowrank} or Theorem~\ref{thm:stable_2implies1_LR} below for the $(L, R)$ lift, and~\cite[\S2.3]{eitan_thesis} for the other lifts in~\eqref{eq:low_rk_lifts}.
\begin{definition} \label{def:2criticalM}
	A point $y \in \calM$ is second-order stationary (or ``2-critical'' for short) for $g$ if, for all smooth curves $c \colon \reals \to \calM$ passing through $y$ at $t = 0$, it holds that $(g \circ c)'(0) = 0$ and $(g \circ c)''(0) \geq 0$.
\end{definition}
When $\calM$ is endowed with a Riemannian metric, 2-criticality is equivalent to
\begin{align}
    \nabla g(y) & = 0 & \textrm{ and } && \nabla^2 g(y) & \succeq 0,
\end{align}
where $\nabla g$ and $\nabla^2 g$ respectively denote the Riemannian gradient and Hessian of $g$.
When $\calM$ is a Euclidean space this recovers the standard second-order necessary optimality conditions for unconstrained optimization.
See~\cite{absil_book,optimOnMans} for further background on Riemannian optimization.

Motivated by the above, we turn our attention to the literature on optimization algorithms with deterministic second-order stationarity guarantees.
The two main algorithms for this task are trust-region methods (TR) and adaptive cubic regularization methods (ARC).
ARC does not explicitly limit the step sizes it considers, which makes it more delicate to ensure all assessed points remain in a compact set and hence have an accumulation point.
For this reason, we favor TR methods.

Common variants of TR find \emph{approximate} 2-critical points (see Definition~\ref{def:2criticalMapproximate} below) in finite time with a preset tolerance on which the algorithm depends~\cite{cartis2012complexity,global_convergence_Nicolas}. This is insufficient for our purpose. We seek an algorithm that, if allowed to run for infinitely-many iterations, produces an infinite sequence accumulating at an actual 2-critical point. 
To our advantage, Curtis et al.~\cite{curtis2018concise} propose a trust-region method for optimization in Euclidean space which does precisely that. 
Below, we generalize their algorithm and its analysis to optimization on Riemannian manifolds.
This is done without friction following~\cite{genrtr,global_convergence_Nicolas}.

Since the algorithm will be terminated after finitely-many steps in practice, we also want to show that approximate 2-critical points for $g$ on $\calM$ map to approximate stationary points for $f$ on $\calX$. As we know from our study of apocalypses, bounding the standard measure of approximate stationarity on $\calX$ in Definition~\ref{def:criticalPapprox} may not be informative. We therefore give a more refined bound in Section~\ref{subsec:approxstationarity}.



Of the three lifts listed in~\eqref{eq:low_rk_lifts}, it might seem that $\varphi(L, R) = LR^\top$ defined on the Euclidean space $\calM = \Rmr \times \Rnr$ is the simplest one, as it does not even require the notion of smooth manifold---perhaps in part for that reason, it is also the most commonly used.
However, its simplicity hides other difficulties.
Specifically, the aforementioned algorithms require $g$ to have properties such as Lipschitz continuous derivatives or compact sublevel sets.
Yet, even if $f$ has those properties, $g$ generally does not.
For example, if $f$ is quadratic then $g(L, R) = f(LR^\top)$ is quartic: this precludes Lipschitz continuity of the gradient (see Dragomir et al.~\cite{dragomir2021quartic} for an interesting take on this issue).
Moreover, the sublevel sets of $g$ are never bounded because the fibers of the lift $\varphi$---that is, the sets of the form $\varphi^{-1}(X)$ for $X\in\Rmnlr$---are unbounded and $g = f \circ \varphi$ is necessarily constant over those fibers.
In contrast, the other two lifts have compact fibers.

Still, because the $(L, R) \mapsto LR^\top$ lift is so common in practice, it is worthwhile to build theory that allows for its use.
To this end, we further generalize Curtis et al.'s algorithm in the following way:
We allow iterates to be modified by a \emph{hook}\footnote{We borrow the term \emph{hook} from computer programming, where it designates the opportunity to inject code into existing code via a function handle; see \url{https://en.wikipedia.org/wiki/Hooking}.} $\Phi \colon \calM \to \calM$, provided this does not increase the value of the cost function.
The resulting method is listed as Algorithm~\ref{algo:TRbalanced}.

The hook allows us to \emph{rebalance} iterates $(L_k, R_k)$ without interfering with convergence (the cost function acts as a Lyapunov function, and that is not affected by hooking).
Rebalancing $(L, R)$ factors is a standard technique to prevent the iterates from drifting off to infinity along fibers of the lift, effectively resolving the issues caused by their unboundedness.\footnote{See also~\cite{ma2021beyondbalancing} for a discussion of situations where factors automatically remain balanced.}
Then, assuming $f$ has bounded sublevel sets we can ensure that iterates on $\calM$ remain in a compact region.
Compactness suffices to ensure some Lipschitz continuity.
In turn, this leads to the global convergence guarantees we need. 

\subsection{A hooked Riemannian trust-region method}

\begin{algorithm}[t]
	\caption{Trust-region method from~\cite{curtis2018concise} extended to manifolds and a hook $\Phi$}
	\label{algo:TRbalanced}
	\begin{algorithmic}[1]
		\State \textbf{Parameters:} $\gamma_c, \eta \in (0, 1)$, $0 < \underline{\gamma} \leq \overline{\gamma} < \infty$.
		\State \textbf{Input:} $y_0 \in \calM$, $\gamma_0 \in [\underline{\gamma}, \overline{\gamma}]$.
		\State Redefine the initial point: $y_0 \leftarrow \Phi(y_0)$. \Comment{Hook initial point.}
		\For{$k$ \textbf{in} $0, 1, 2, \ldots$}
		\If{$\lambda_k < 0$ \textbf{and} $\|\nabla g(y_k)\|^2 < |(\lambda_k)_-|^3$} 
		\State \label{step:deltalambda} Set $\delta_k = \gamma_k |(\lambda_k)_-|$.
		\Else
		\State \label{step:deltagrad} Set $\delta_k = \gamma_k \|\nabla g(y_k)\|$.
		\EndIf
		\State \label{step:sk} Let $s_k \in \T_{y_k} \calM$ satisfy $\|s_k\| \leq \delta_k$ and $m_k(s_k) \leq m_k(s_k^c)$, e.g., as in~\eqref{eq:skc}.
		\State \label{step:rhok} Compute $\rho_k$ as follows (if the denominator is zero, set $\rho_k = 1$ instead):
		\begin{align}
			\rho_k & = \frac{g(y_k) - g(\Retr_{y_k}(s_k))}{m_k(0) - m_k(s_k)}.
			\label{eq:rhok}
		\end{align}
		\If{$\rho_k \geq \eta$}
		\State \label{step:sucessrebalanced} Set $y_{k+1} = \Phi(\Retr_{y_k}(s_k))$ and choose any $\gamma_{k+1} \in [\underline{\gamma}, \overline{\gamma}]$.  \Comment{Hook new point.}
		\Else
		\State Set $y_{k+1} = y_k$ and $\gamma_{k+1} = \gamma_c \gamma_k$.
		\EndIf
		\EndFor
	\end{algorithmic}
\end{algorithm}

Algorithm~\ref{algo:TRbalanced} is essentially the non-standard trust-region method developed by Curtis et al.~\cite{curtis2018concise}, further modified so that it can operate on Riemannian manifolds and such that it allows a hook $\Phi \colon \calM \to \calM$ in Step~\ref{step:sucessrebalanced} (and at initialization) to modify the sequence of iterates.
Setting $\Phi(y) = y$ removes the effect of hooking, but then our guarantees may not apply---see Proposition~\ref{prop:propervarphigoodPhi} below.
At least, we require the following from $\Phi$:
\begin{assumption} \label{assu:Phidecrease}
    The hook $\Phi \colon \calM \to \calM$ satisfies $g(\Phi(y)) \leq g(y)$ for all $y \in \calM$.
\end{assumption}

This section studies the behavior of Algorithm~\ref{algo:TRbalanced} applied to the general problem $\min_{y\in\calM} g(y)$ assuming $g$ is twice continuously differentiable.
We review the essential steps of the analysis in~\cite{curtis2018concise}, with the necessary comments and (light) modifications necessary to accommodate the proposed extensions.
Remarkably, the original proofs apply mostly as is: we do not repeat them.

Given an initial point $y_0 \in \calM$, Algorithm~\ref{algo:TRbalanced} generates a sequence $(y_k)_{k \geq 0}$ on $\calM$.
To do so, it uses a \emph{retraction} $\Retr$, that is, a smooth map from the tangent bundle of $\calM$ to $\calM$ with the property that $c(t) = \Retr_y(tu)$ is a smooth curve on $\calM$ satisfying $c(0) = y$ and $c'(0) = u$~\cite[\S4.1]{absil_book}.
For $\calM$ a linear space, it is convenient to use $\Retr_y(u) = y+u$.

We endow $\calM$ with a Riemannian metric, so that each tangent space $\T_y\calM$ is endowed with an inner product $\inner{\cdot}{\cdot}_y$ together with an associated norm $\|\cdot\|_y$.
We typically omit subscripts when referring to these objects, thus writing simply $\inner{\cdot}{\cdot}$ and $\|\cdot\|$.
Context distinguishes them from their counterparts on the Euclidean space $\calE$.
We let $\nabla g$ and $\nabla^2 g$ denote the Riemannian gradient and Hessian of $g$, respectively.

At iterate $y_k \in \calM$, we consider the \emph{pullback} $\hat g_k = g \circ \Retr_{y_k} \colon \T_{y_k} \calM \to \reals$.
This is a function on a Euclidean space.
We approximate it with a second-order Taylor expansion, yielding a quadratic model $m_k \colon \T_{y_k}\calM \to \reals$:
\begin{align}
	m_k(s) & = g(y_k) + \innersmall{s}{\nabla \hat g_{k}(0)} + \tfrac{1}{2} \innersmall{s}{\nabla^2 \hat g_{k}(0)[s]}.
\end{align}
It is helpful to note that $\nabla \hat g_k(0) = \nabla g(y_k)$ and, if the retraction is second order~\cite[Prop.~5.5.5]{absil_book}, also $\nabla^2 \hat g_k(0) = \nabla^2 g(y_k)$.
Moreover, we write
\begin{align}
	\lambda_k & = \lambdamin(\nabla^2 \hat g_k(0)) & \textrm{ and }  && (\lambda_k)_- & = \min(\lambda_k, 0)
\end{align}
to denote the left-most eigenvalue of $\nabla^2 \hat g_k(0)$ and its negative part.

The iterations of Algorithm~\ref{algo:TRbalanced} are of two types:
\begin{align}
	\calK_1 & = \{ k \in \mathbb{N} : \delta_k \textrm{ is set by Step~\ref{step:deltagrad}} \}, &
	\calK_2 & = \{ k \in \mathbb{N} : \delta_k \textrm{ is set by Step~\ref{step:deltalambda}} \}.
\end{align}
Depending on the type of iteration of $k$, we define a different \emph{Cauchy step} $s_k^c$.
Specifically, let $e_k \in \T_{y_k} \calM$ denote a unit-norm eigenvector of $\nabla^2 \hat g_k(0)$ associated to its least eigenvalue $\lambda_k$, with sign chosen so that $\inner{e_k}{\nabla g(y_k)} \leq 0$.
Then, with
\begin{align}
	u_k & = \begin{cases}
		-\nabla g(y_k) & \textrm{ if } k \in \calK_1, \\
		 e_k & \textrm{ if } k \in \calK_2,
	\end{cases}
	\label{eq:uk}
\end{align}
we let the Cauchy step be defined by 
\begin{align}
	s_k^c & = t_k u_k & \textrm{ where } && t_k & \in \argmin{t \geq 0}\ m_k(tu_k) \textrm{ subject to } \|tu_k\| \leq \delta_k.
	\label{eq:skc}
\end{align}

This completes the notation required to interpret Algorithm~\ref{algo:TRbalanced}.
Our analysis follows in the steps of Curtis et al.~\cite{curtis2018concise}, incorporating elements of the papers~\cite{global_convergence_Nicolas,agarwal2018arcfirst} to handle the Riemannian extension.
Details are in Appendix~\ref{appdx:hookedRTR}.

We make a regularity assumption on the composition of $g$ with the retraction---see also Corollary~\ref{cor:optimassumptionlipschitzhelper} in the appendix which gives a sufficient condition for these assumptions.
In the Euclidean case, the required properties hold if we assume Lipschitz continuity of the gradient and Hessian.
This extends to the Riemannian case~\cite[\S10.4, \S10.7]{optimOnMans}.
\begin{assumption} \label{assu:upstairs}
	There exist $L_1, L_2 \geq 0$ such that, for all $(y_k, s_k)$ generated by Algorithm~\ref{algo:TRbalanced}, with $\hat g_k = g \circ \Retr_{y_k}$, we have $\|\nabla^2 \hat g_k(0)\| \leq L_1$ (in operator norm) and
	\begin{align}
		\left| g(\Retr_{y_k}(s_k)) - g(y_k) - \inner{s_k}{\nabla g(y_k)} \right| & \leq \frac{L_1}{2} \|s_k\|^2, \\
		\left| g(\Retr_{y_k}(s_k)) - g(y_k) - \inner{s_k}{\nabla g(y_k)} - \frac{1}{2} \inner{s_k}{\nabla^2 \hat g_k(0)[s_k]} \right| & \leq \frac{L_2}{6} \|s_k\|^3.
	\end{align}
\end{assumption}
The main result of this subsection follows.
\begin{theorem} \label{thm:optimtheorem1}
	Given a sequence $(y_k)_{k \geq 0}$ generated by Algorithm~\ref{algo:TRbalanced}, let
	\begin{align}
		K(\epsilon_1, \epsilon_2) & = \left|\{ k : \|\nabla g(y_k)\| > \epsilon_1 \textrm{ or } |(\lambda_k)_-| > \epsilon_2 \}\right|.
	\end{align}
	Under Assumptions~\ref{assu:Phidecrease} and~\ref{assu:upstairs}, and further assuming there exists $\gmin \in \reals$ such that $g(y_k) \geq \gmin$ for all $k$, we have for all $\epsilon_1, \epsilon_2 > 0$:
	\begin{align}
		K(\epsilon_1, \epsilon_2) & \leq \ceil{\log_{\gamma_c}\!\left(\frac{\gammamin}{\overline{\gamma}}\right) + 1} \floor{\frac{g(y_0) - \gmin}{\kappamin} \max\!\left( \frac{1}{\epsilon_1^2}, \frac{1}{\epsilon_2^3} \right) },
	\end{align}
    where we use the constants $\gammamin$ and $\kappamin$ defined (with $L_1, L_2$ as in Assumption~\ref{assu:upstairs}):
    \begin{align}
        \gammamin & = \min\!\left( \underline{\gamma}, \frac{\gamma_c}{1+L_1}, \frac{\gamma_c(1-\eta)}{2L_1}, \frac{3\gamma_c(1-\eta)}{L_2} \right) \in (0, 1), 
        \label{eq:gammamin} \\
        \kappamin & = \frac{1}{2} \eta \gammamin^2.
        \label{eq:kappamin}
    \end{align}
    All accumulation points of $(y_k)_{k \geq 0}$ (if any) are 2-critical points of $\min_{y\in\calM} g(y)$.
\end{theorem}
%
If we can compute left-most eigenvectors of symmetric linear maps (thus providing $e_k$ in~\eqref{eq:uk}), then the Cauchy step $s_k^c$~\eqref{eq:skc} can be computed explicitly and we can set $s_k = s_k^c$ to meet the provisions of the analysis.
This is not required however: it is sufficient for $s_k$ to satisfy $\|s_k\| \leq \delta_k$ and $m_k(s_k) \leq m_k(s_k^c)$.
Moreover, the analysis can be relaxed to accommodate approximate left-most eigenvalue/eigenvector computation.
There exist many practical methods to compute trial steps in trust-region methods, see for example~\cite{conn2000trust}.
\begin{tbd} 

If we can compute left-most eigenvectors of symmetric linear maps (thus providing $e_k$ in~\eqref{eq:uk}), then the Cauchy step $s_k^c$~\eqref{eq:skc} can be computed explicitly and we can set $s_k = s_k^c$ to meet the provisions of the analysis.
This is not required however: it is sufficient for $s_k$ to satisfy $\|s_k\| \leq \delta_k$ and $m_k(s_k) \leq m_k(s_k^c)$.
Moreover, the analysis can be relaxed to accommodate approximate left-most eigenvalue/eigenvector computation.
There exist many practical methods to compute trial steps in trust-region methods, see for example~\cite{conn2000trust}.
\end{tbd}

\subsection{Rebalancing maps as hooks}\label{sec:rebalancing_maps}

The previous section provides an algorithm to find 2-critical points of~\eqref{eq:Q} under Assumptions~\ref{assu:Phidecrease} and~\ref{assu:upstairs}.
It remains to show that those assumptions can be met under reasonable conditions on~\eqref{eq:P}. 
Standard arguments show that Assumption~\ref{assu:upstairs} holds if $g$ is three times continuously differentiable and the iterates remain in a compact set, see Corollary~\ref{cor:optimassumptionlipschitzhelper} in Appendix~\ref{apdx:lip_assump}. In this section, we identify a class of hooks $\Phi$ that guarantee the latter condition while satisfying Assumption~\ref{assu:Phidecrease}. We begin by identifying a set containing all the iterates of Algorithm~\ref{algo:TRbalanced}.

\begin{theorem} \label{thm:optimthmupstairs}
	If $g$ is three times continuously differentiable and the hook satisfies $\varphi \circ \Phi = \varphi$,
    then the sequence $(y_k)_{k \geq 0}$ generated by Algorithm~\ref{algo:TRbalanced} remains in the set
	\begin{align}
		\calL & = \Phi\!\left( \{ y \in \calM : g(y) \leq g(y_0) \} \right) = \Phi\!\left(\varphi^{-1}\!\left( \left\{ x \in \calX : f(x) \leq f(x_0) \right\} \right) \right), \label{eq:calL}
	\end{align}
	where $x_0 = \varphi(y_0)$.
	If $\calL$ is contained in a compact set, then 
	$(y_k)$ admits at least one accumulation point, and any such point
	is 2-critical for~\eqref{eq:Q}.
\end{theorem}
\begin{proof}
    Assumption~\ref{assu:Phidecrease} holds because $g \circ \Phi = f \circ (\varphi \circ \Phi) = f \circ \varphi = g$.
	By design, $y_0$ is in $\calL$. 
	Algorithm~\ref{algo:TRbalanced} guarantees that $y_{k+1}$ is either equal to $y_k$ or set to $\Phi(y)$ for some $y$ which satisfies $g(y) \leq g(y_k)$.
	By induction, it follows that $(y_k)$ is included in $\Phi\!\left( \{ y \in \calM : g(y) \leq g(y_0) \} \right)$, which is the definition of $\calL$.
	Now assume $\calL$ is contained in a compact set.
	Assumption~\ref{assu:upstairs} then holds by Corollary~\ref{cor:optimassumptionlipschitzhelper}.
	It follows that $g(y)$ is lower-bounded by some finite $\gmin$ for all $y \in \calL$, therefore also for all $y \in (y_k)$.
	By Bolzano--Weierstrass, $(y_k)$ has an accumulation point. 
	Apply Theorem~\ref{thm:optimtheorem1} to verify that all accumulation points of $(y_k)$ are 2-critical for~\eqref{eq:Q}.
\end{proof}

Motivated by Theorem~\ref{thm:optimthmupstairs}, we introduce a helpful class of hooks we call \emph{rebalancing maps}.
Note that we do \emph{not} require $\Phi$ to be continuous, \emph{nor} do we require $\Phi^2 = \Phi$.
In words, the conditions below require that (a) rebalancing a point on $\calM$ does not change its projection through $\varphi$, and (b) rebalanced points cannot drift off to infinity along fibers of $\varphi$.
\begin{definition} \label{def:goodPhi}
	The map $\Phi \colon \calM \to \calM$ is a \emph{rebalancing map} with respect to a lift $\varphi \colon \calM \to \calE$ of $\calX = \varphi(\calM)$ if
	\begin{enumerate}
		\item[(a)] $\varphi \circ \Phi = \varphi$, and
		\item[(b)] $\Phi(\varphi^{-1}(B))$ is bounded for all bounded $B \subseteq \calX$.
	\end{enumerate}
	(Boundedness on $\calM$ and $\calE$ is assessed with respect to Riemannian and Euclidean distances, respectively.)
\end{definition}
\begin{lemma}
	Assume $\calM$ is complete, $\Phi$ is a rebalancing map with respect to $\varphi$ and
    $\{ x \in \calX : f(x) \leq f(x_0) \}$ is bounded.
    Then the closure of $\calL$~\eqref{eq:calL} is compact.
\end{lemma}
\begin{proof}
	The set $\calL$ is bounded since $B = \{ x \in \calX : f(x) \leq f(x_0) \}$ is bounded, $\calL = \Phi(\varphi^{-1}(B))$ and $\Phi$ is a rebalancing map.
	It follows that
    the closure of $\calL$
    is
    a compact set (since $\calM$ is complete) which contains $\calL$.
\end{proof}
The above yields the main result of this subsection as a corollary.
\begin{corollary} \label{cor:optimresult}
	Assume $\calM$ is complete, $\Phi$ is a rebalancing map with respect to $\varphi$, the sublevel set $\{ x \in \calX : f(x) \leq f(x_0) \}$ is bounded
	and $f$ is three times continuously differentiable.
	Then sequences produced by Algorithm~\ref{algo:TRbalanced}
	with $y_0 \in \varphi^{-1}(x_0)$ 
	admit at least one accumulation point, and any such point is 2-critical for~\eqref{eq:Q}.
\end{corollary}

Sometimes, we do not need rebalancing and can set $\Phi(y) = y$.
This is the case exactly if $\varphi$ is \emph{proper}, that is, if pre-images of compact sets are compact.
For example, this holds when $\calM$ itself is compact.
The following lemma captures that observation.
Intuitively, properness helps because if $f$ has compact sublevel sets then so does $g$.
\begin{proposition} \label{prop:propervarphigoodPhi}
	Assume $\calM$ is complete.
	The identity map $\Phi(y) = y$ is a rebalancing map for $\varphi \colon \calM \to \calE$ if and only if $\varphi$ is proper.
\end{proposition}
\begin{proof}
	Clearly, $\varphi \circ \Phi = \varphi$.
	Assume $\Phi$ is a rebalancing map.
	Let $K$ be an arbitrary compact set in $\calX$.
	In particular, $K$ is closed and bounded.
	It follows that $\varphi^{-1}(K)$ is closed (by continuity of $\varphi$) and bounded (because $\Phi$ is a rebalancing map).
	Since $\calM$ is complete, we have that $\varphi^{-1}(K)$ is compact.
	We conclude that $\varphi$ is proper.
	Conversely, assume $\varphi$ is proper.
	Let $B$ be an arbitrary bounded set in $\calX$.
	There exists a compact set $K \subseteq \calX$ which contains $B$. 
	Therefore, $\varphi^{-1}(B) \subseteq \varphi^{-1}(K)$, and $\varphi^{-1}(K)$ is compact since $\varphi$ is proper.
	It follows that $\varphi^{-1}(B) = \Phi(\varphi^{-1}(B))$ is bounded.
	We conclude that $\Phi$ is a rebalancing map.
\end{proof}

We close this subsection with a philosophical point.
Algorithm~\ref{algo:TRbalanced} generates a sequence $(y_k)$ on the manifold $\calM$ and a sequence of scalars $(\gamma_k)$ which control the trust-region radius.
Formally, the algorithm (together with a memory-less subproblem solver) implements a map
\begin{align*}
	G \colon \calM \times \reals \to \calM \times \reals
\end{align*}
such that $(y_{k+1}, \gamma_{k+1}) = G(y_k, \gamma_k)$.
Through the lift $\varphi$, the sequence $(y_k)$ on $\calM$ defines a sequence $(x_k)$ on $\calX$ via $x_k = \varphi(y_k)$.
One may ask: is there an algorithm operating on $\calX$ directly which generates $(x_k)$?
More formally: is there a map
\begin{align*}
	F \colon \calX \times \reals \to \calX \times \reals
\end{align*}
such that $(x_{k+1}, \gamma_{k+1}) = F(x_k, \gamma_k)$?
The answer is yes if the rebalancing map $\Phi$ maps fibers to singletons, that is, if $\Phi(\varphi^{-1}(x))$ is a single point $y$ of $\calM$ such that $\varphi(y) = x$.
Indeed, we then have
\begin{align*}
	F(x, \gamma) & = (\varphi(\tilde y), \tilde \gamma) & \textrm{ with } && (\tilde y, \tilde \gamma) = G(\Phi(\varphi^{-1}(x)), \gamma).
\end{align*}
Because of this interesting property, we give a name to rebalancing maps which are constant on fibers.
\begin{definition}
    A \emph{sectional rebalancing map} for $\varphi \colon \calM \to \calE$ is a rebalancing map $\Phi$ with the additional property that $\Phi(\varphi^{-1}(x))$ is a singleton for all $x \in \calX$.
\end{definition}
In the next section, the rebalancing map we propose for the $LR^\top$ lift is sectional.
One could then ask: what kind of algorithm on $\calX$ does the map $F$ encode?
A closer look suggests that it is a type of second-order algorithm on the nonsmooth set $\calX$ which considers the Hessian of $f$ only along specific subspaces of the tangent cones.
Our algorithm uses this additional information to escape apocalypses.

\subsection{Examples of rebalancing maps, and proof of Theorem~\ref{thm:intro_informal_alg}}\label{sec:examples}

We construct rebalancing maps for all three lifts
in~\eqref{eq:low_rk_lifts}, where $\calX = \Rmnlr$
is a subset of $\calE = \Rmn$ with its usual inner product $\inner{U}{V} = \Tr(U^\top V)$.
We can then use Corollary~\ref{cor:optimresult} to prove Theorem~\ref{thm:intro_informal_alg} from Section~\ref{sec:intro}.

\begin{example} \label{ex:LR}
	Consider $\calM = \Rmr \times \Rnr$ with its usual Euclidean structure and $\varphi(L, R) = LR\transpose$.
	Note that $\calM$ is complete.
	Further let $\Phi(L, R) = (U\Sigma^{1/2}, V\Sigma^{1/2})$ where $U\Sigma V\transpose$ is an SVD of $\varphi(L, R)$ with $\Sigma$ of size $r \times r$ (if it is not unique, pick one deterministically.) 
	This is a sectional rebalancing map for $\varphi$.
	Indeed, $\varphi \circ \Phi = \varphi$ by design; $\Phi$ maps fibers of $\varphi$ to singletons; and the boundedness property is ensured by
    \begin{align*}
        \|\Phi(L, R)\|^2 & = \|U\Sigma^{1/2}\|^2 + \|V\Sigma^{1/2}\|^2 = 2\Tr(\Sigma) = 2\|\varphi(L, R)\|_* \leq 2\sqrt{r}\|\varphi(L, R)\|,
    \end{align*}
	where $\|\cdot\|_*$ denotes the nuclear norm on $\Rmn$.
\end{example}
Using the rebalancing map in Example~\ref{ex:LR} has an additional advantage. In general, a local minimum $(L, R)$ of $g$ on $\calM$ may map to a saddle point $\varphi(L, R)$ of $f$ on $\Rmnlr$~\cite[Prop.~2.30]{eitan_thesis}.
Fortunately, with the stated rebalancing map we are guaranteed that if Algorithm~\ref{algo:TRbalanced} converges to a local minimum $(L, R)$ of $g$ then $\varphi(L, R)$ is a local minimum of $f$.
This follows from~\cite[Lem.~2.33, Prop.~2.34]{eitan_thesis}: if $L\transpose L = R\transpose R$---which is enforced by the above rebalancing map at every iteration and is closed under limits---then $(L, R)$ is a local minimum of $g$ on $\calM$ if and only if $\varphi(L, R)$ is a local minimum for $f$ on $\Rmnlr$.

For the other two examples of lifts of $\Rmnlr$ from~\eqref{eq:low_rk_lifts}, we show that the lift is proper, and therefore no rebalancing is needed by Proposition~\ref{prop:propervarphigoodPhi}.
To this end, it is helpful to recall the following facts.
We say a sequence $(x_k)$ in a topological space \emph{escapes to infinity} if for each compact $K$ there exists $\bar k$ such that $x_k\notin K$ for all $k \geq \bar k$.
The spaces $\calM$ and $\calX$ are metric spaces ($\calM$ with its Riemannian distance, $\calX$ as a subset of $\calE$) and $\varphi$ is continuous.
Therefore, it holds that $\varphi$ is proper if and only if, whenever $(y_k)$ escapes to infinity in $\calM$, the corresponding sequence $(x_k)$ with $x_k = \varphi(y_k)$ escapes to infinity in $\calX$.

\begin{example}
	Consider $\calM = \St(m, r) \times \Rnr$ endowed with the Riemannian product structure, where $\St(m, r) = \{ U \in \Rmr : U^\top U = I_r \}$ has the usual Riemannian submanifold structure in $\Rmr$.
	Then, $\calM$ is complete because it is a product of complete manifolds.
	Let $\varphi(U, W) = UW\transpose$, which is clearly continuous.
	The identity $\Phi(U, W) = (U, W)$ is a rebalancing map for $\varphi$ by Proposition~\ref{prop:propervarphigoodPhi}.
	Indeed, to see that $\varphi$ is proper, suppose a sequence $((U_k, W_k))_{k\geq0}$ on $\calM$ escapes to infinity. Since $U_k$ lives on a compact set (the Stiefel manifold), $(U_k, W_k)$ escapes to infinity exactly when $(W_k)$ escapes to infinity in $\Rmr$.
	Since $\|X_k\| = \|U_k^{}W_k^\top\| = \|W_k\|$, this happens exactly when $(X_k)$ escapes to infinity.

	We could also consider the Riemannian quotient manifold $\calM \slash \Or$ through the action $((U, W), Q) \mapsto (UQ, WQ)$ of the orthogonal group $\Or$~\cite[\S9.7]{optimOnMans}.
	The quotient is complete because $\calM$ is complete; $\varphi$ is well defined on the quotient; and $\Phi=\id$ is still a rebalancing map because distances only decrease when passing to the quotient.
\end{example}

\begin{example}
	Let $\Gr(n, n-r)$ denote the manifold of subspaces of dimension $n-r$ in $\Rn$, with its usual Riemannian structure~\cite{bendokat2020grassmann}.
	Consider $\calM = \{ (X, K) \in \Rmn \times \Gr(n, n-r) : K \subseteq \ker(X) \}$ endowed with the Riemannian submanifold structure in $\Rmn \times \Gr(n, n-r)$, which itself has the usual product structure.
	Then, $\calM$ is complete because it is closed in, and a Riemannian submanifold of, a complete manifold.
	Let $\varphi(X, K) = X$.
	The identity $\Phi(X, K) = (X, K)$ is a rebalancing map for $\varphi$ by Proposition~\ref{prop:propervarphigoodPhi}.
	Indeed, $\varphi$ is proper since it is continuous and any sequence $((X_k, K_k))_{k\geq0}$ on $\calM$ which escapes to infinity has the property that the corresponding sequence $(X_k)_{k\geq0}$ (where $X_k = \varphi(X_k, K_k)$) escapes to infinity.
	That is because $K_k$ is an element of the Grassmann manifold which is compact.

\end{example}

Using any of the above rebalancing maps, we can now prove Theorem~\ref{thm:intro_informal_alg} from Section~\ref{sec:intro}.
\begin{proof}[Proof of Theorem~\ref{thm:intro_informal_alg}]
    Apply Algorithm~\ref{algo:TRbalanced} to problem~\eqref{eq:Q} with any of the three lifts in~\eqref{eq:low_rk_lifts}, using the corresponding rebalancing map $\Phi$ from the above examples. By Corollary~\ref{cor:optimresult} and the continuity of $\varphi$, the sequence $(Y_k)_{k\geq 1}\in\calM$ generated by Algorithm~\ref{algo:TRbalanced} has at least one accumulation point $Y_*$, and any such point is 2-critical for~\eqref{eq:Q}. By~\cite[Thm.~1]{ha2020criticallowrank} and~\cite[\S2.3]{eitan_thesis}, the image of a 2-critical point $X_*=\varphi(Y_*)\in\Rmnlr$ is stationary for~\eqref{eq:P}. By continuity of $\varphi$, the matrix $X_*$ is an accumulation point of the sequence $X_k=\varphi(Y_k)$, which satisfies the conclusion of Theorem~\ref{thm:intro_informal_alg}.
\end{proof}


We close this subsection with a consideration of low-rank tensors.

\begin{example}\label{ex:tensors}
A natural family of lifts and sectional rebalancing maps are afforded by tensor networks with underlying tree structure~\cite[Ch.~11--12]{hackbusch2012tensor}. 
Here $\varphi$ is a certain multilinear map specified by a graph, the domain of $\varphi$ is a Cartesian product of tensor products of Euclidean spaces (one tensor product per node in the graph), and the codomain is a (generally) larger tensor product of Euclidean spaces. 
In the case that the underlying graph is a tree, the image of $\varphi$ is a closed subset in the Euclidean topology, a real algebraic variety but typically with singularities.
A sectional rebalancing map is then given by the sequential matrix singular value decomposition algorithm described in~\cite[\S12.4.1]{hackbusch2012tensor}.  
In particular, this family of examples contains two well-known low-rank tensor representations with corresponding normalization algorithms: the Tensor Tucker decomposition~\cite{de2000multilinear} with the Higher-Order Singular Value Decomposition (HOSVD), and the Tensor Train Decomposition~\cite{oseledets2011tensor} with the Tensor Train Singular Value Decomposition (TT-SVD).
We work out the details for these two examples in Appendix~\ref{apdx:tensors}, and show there that the HOSVD and TT-SVD are indeed sectional rebalancing maps. 
\begin{journal}
We work out the details for these two examples in~\cite[App.~D]{levin2021finding}, and show there that the HOSVD and TT-SVD are indeed sectional rebalancing maps. 
\end{journal}
\begin{tbd}
We work out the details for these two examples in Appendix~\ref{apdx:tensors}, and show there that the HOSVD and TT-SVD are indeed sectional rebalancing maps. 
\end{tbd}
While Corollary~\ref{cor:optimresult} applies to these lifts, unfortunately we are not able to prove an analog of Theorem \ref{thm:intro_informal_alg} here, because not all 2-critical points for the lifted optimization problem map to stationary points on the tensor variety.
Thus, provably finding stationary points on bounded Tucker, or Tensor Train, rank varieties remains an open problem.
\end{example}

\subsection{From approximate 2-criticality to approximate stationarity}
\label{subsec:approxstationarity}


We define the usual notion of approximate 2-critical point for $g \colon \calM \to \reals$ by relaxing Definition~\ref{def:2criticalM} using the Riemannian structure.
\begin{definition} \label{def:2criticalMapproximate}
    A point $y \in \calM$ is \emph{$(\epsilon_1, \epsilon_2)$-approximate 2-critical} for $g \colon \calM \to \reals$ with $\epsilon_1, \epsilon_2 \geq 0$ if
    \begin{align*}
        \|\nabla g(y)\| & \leq \epsilon_1 & \textrm{ and } && \lambdamin(\nabla^2 g(y)) & \geq -\epsilon_2.
    \end{align*}
\end{definition}
Theorem~\ref{thm:optimtheorem1} bounds the number of iterations Algorithm~\ref{algo:TRbalanced} may need to compute an $(\epsilon_1, \epsilon_2)$-approximate 2-critical point on the manifold.
It also shows that any accumulation point of the sequence generated by the algorithm is an exact 2-critical point, or equivalently, a $(0,0)$-approximate 2-critical point.
We know those map to stationary points on $\Rmnlr$ for the lifts in~\eqref{eq:low_rk_lifts}.
It is natural to ask whether approximate 2-critical points map to approximate stationary points. 

We provide guarantees for the $(L, R)$ lift.
(Analogous arguments for the other two lifts in~\eqref{eq:low_rk_lifts} give similar results.)
Merely bounding $\|\Proj_{\T_X\Rmnlr}(-\nabla f(X))\|$, the stationarity measure in Definition~\ref{def:criticalPapprox}, may not be informative if $\sigma_r(X)$ is small
since all lesser rank matrices are apocalyptic. There may not be any stationary point nearby even if that measure is small.
To resolve this, we also bound the entire Euclidean gradient norm $\|\nabla f(X)\|$. 

Theorem~\ref{thm:stable_2implies1_LR} below generalizes a result by Ha, Liu and Barber~\cite{ha2020criticallowrank} to $\epsilon_1, \epsilon_2$ nonzero. 
Our proof is a direct modification of theirs.
The rebalancing map $\Phi$ we use for the $(L, R)$ lift (Example~\ref{ex:LR}) produces iterates $(L, R)$ satisfying $L^\top L = R^\top R$.
We call such iterates \emph{balanced} and use this property below.
The bounds below are sharp up to constants, which we show with a simple example in Appendix~\ref{apdx:stable_LR_sharp}.
\begin{journal} 
The bounds below are sharp up to constants, which we show with a simple example in~\cite[App.~E]{levin2021finding}.
\end{journal}
\begin{tbd}
The bounds below are sharp up to constants, which we show with a simple example in Appendix~\ref{apdx:stable_LR_sharp}.
\end{tbd}
\begin{theorem}\label{thm:stable_2implies1_LR}
    Suppose $f\colon\Rmn\to\RR$ has $L_f$-Lipschitz continuous gradient. Suppose $(L,R)$ is an $(\epsilon_1,\epsilon_2)$-approximate 2-critical point for $g(L,R) = f(LR^\top )$ satisfying $L^\top L = R^\top R$, and let $X = LR^\top$ be the corresponding point in $\Rmnlr$.
    Then,
    \begin{align}
        \|\nabla f(X)\|_{\mathrm{op}} \leq \epsilon_2 + 2L_f\sigma_r(X),
        \label{eq:opnormnablafXbound}
    \end{align}
    where $\|\cdot\|_{\mathrm{op}}$ denotes the largest singular value of a matrix.
    Moreover,
    \begin{align}
        \|\Proj_{\T_X\Rmnlr}(-\nabla f(X))\| \leq \min\left\{\sqrt{\frac{2}{\sigma_r(X)}} \epsilon_1,\ \sqrt{\rank(\nabla f(X))}(\epsilon_2 + 2L_f\sigma_r(X))\right\},
        \label{eq:normProjTXnegnablabound}
    \end{align}
    where $\epsilon_1/\sqrt{\sigma_r(X)}$ is interpreted as $+\infty$  if $\rank(X) < r$ (even if $\epsilon_1 = 0$).
\end{theorem}
\begin{journal}
The proof of Theorem~\ref{thm:stable_2implies1_LR} is a straightforward modification of the proof of~\cite[Thm.~1(a)]{ha2020criticallowrank}. We give the full proof in~\cite[Thm.~3.16]{levin2021finding}.
\end{journal}
\begin{tbd}
Our proof is a direct modification of theirs.
\end{tbd}
\begin{proof}
All balanced factorizations $(L,R)$ of $X=LR^\top$ are of the form $L=U\Sigma^{1/2}Q$ and $R=V\Sigma^{1/2}Q^\top$ where $X=U\Sigma V^\top$ is a thin SVD of $X$ and $Q\in O(r)$ is arbitrary~\cite[Lem.~2.33]{eitan_thesis}. 
Assume first that $\sigma_r(X) > 0$.
Since $(L, R)$ is balanced, we have $\|L^{\dagger}\|_{\mathrm{op}} = \|R^{\dagger}\|_{\mathrm{op}} = \sigma_r(X)^{-1/2}$.
Because $(L, R)$ is an $(\epsilon_1, \epsilon_2)$-approximate 2-critical point for $g$, it holds in particular that
\begin{align*}
   \sqrt{\|L^\top\nabla f(X)\|^2 + \|\nabla f(X)R\|^2} = \|\nabla g(L,R)\| \leq \epsilon_1.
\end{align*}
Since $\sigma_r(X)>0$, we have $\rank(X)=\rank(L)=\rank(R)=r$, so $\mathrm{col}(X)=\mathrm{col}(L)$ and $\mathrm{row}(X)=\mathrm{col}(R)$. Therefore,
\begin{align*}
    &\textrm{if } \nabla f(X) = U\begin{bmatrix} M_1 & M_2\\ M_3 & M_4\end{bmatrix}V^\top,\textrm{ then } LL^{\dagger}\nabla f(X) = \Proj_{\mathrm{col}(X)}\nabla f(X) = U\begin{bmatrix} M_1 & M_2 \\ 0 & 0\end{bmatrix}V^\top,\\ &\nabla f(X)RR^{\dagger} = \nabla f(X)\Proj_{\mathrm{row}(X)} = U\begin{bmatrix} M_1 & 0\\ M_3 & 0\end{bmatrix}V^\top, \textrm{ and } LL^{\dagger}\nabla f(X)RR^{\dagger} = U\begin{bmatrix} M_1 & 0\\ 0 & 0\end{bmatrix}V^\top.
\end{align*}
By~\eqref{eq:proj_Tx} we get:
\begin{align*}
    \|\Proj_{\T_X\Rmnlr}(-\nabla f(X))\| & = \|LL^\dagger\nabla f(X) + \nabla f(X)RR^{\dagger} - LL^\dagger\nabla f(X)RR^{\dagger}\| \\
            & \leq \|(L^{\dagger})^\top L^\top \nabla f(X) + \nabla f(X)RR^{\dagger}\| \\
            & \leq \sigma_r(X)^{-1/2}(\|L^\top\nabla f(X)\|+\|\nabla f(X)R\|) 
            \leq \sqrt{\frac{2}{\sigma_r(X)}} \epsilon_1,
\end{align*}
giving us one of the claimed bounds.

Next, assume only $\rank(X) \leq r$.
Since $(L, R)$ is balanced we can pick $w \in \RR^r$ such that $\|Lw\| = \|Rw\| = \sigma_r(X)^{1/2}$ and $\|w\| = 1$.
Let $u_1, v_1$ be top left and right singular vectors of $\nabla f(X)$, so $\|u_1\| = \|v_1\| = 1$ and $\nabla f(X)v_1 = \|\nabla f(X)\|_{\mathrm{op}} u_1$.
Set $\dot L = u_1w^\top$ and $\dot R = -v_1w^\top$.
Since $f$ has an $L_f$-Lipschitz continuous gradient, the Hessian $\nabla^2 f(X)$ has operator norm bounded by $L_f$.
Thus,
\begin{align*}
    \langle\nabla^2f(X)[\dot LR^\top\! +L\dot R^\top],\dot LR^\top\! + L\dot R^\top \rangle & \leq L_f\|\dot LR^\top\! +L\dot R^\top\|^2 \\
                   & \leq L_f\Big(\|u_1(Rw)^\top\| + \|Lwv_1^\top\|\Big)^2 
                   \leq 4L_f\sigma_r(X).
\end{align*}
We also have
\begin{align*}
    \langle\nabla f(X),\dot L\dot R^\top \rangle = -\|\nabla f(X)\|_{\mathrm{op}},
\end{align*}
and $\|\dot L\|=\|\dot R\|=1$.
By approximate 2-criticality and the above estimates, we get
\begin{align*}
    -2\epsilon_2 = -\epsilon_2(\|\dot L\|^2\!+\!\|\dot R\|^2) &\leq \langle \nabla^2g(L,R)[\dot L,\dot R],[\dot L,\dot R]\rangle\\
    &= \langle\nabla^2f(X)[\dot LR^\top\! + L\dot R^\top],\dot LR^\top\! + L\dot R^\top \rangle + 2\langle\nabla f(X),\dot L\dot R^\top \rangle\\
    &\leq 4L_f\sigma_r(X)-2\|\nabla f(X)\|_{\mathrm{op}},
\end{align*}
or equivalently,
\begin{align}
    \|\nabla f(X)\|_{\mathrm{op}} \leq \epsilon_2 + 2L_f\sigma_r(X).
\end{align}
This is the first claimed bound. 
Finally, $\|\Proj_{\T_X\Rmnlr}(-\nabla f(X))\|\leq \|\nabla f(X)\|\leq\sqrt{\rank(\nabla f(X))}\|\nabla f(X)\|_{\mathrm{op}}$ where the first inequality holds by Proposition~\ref{prop:normprojtocone}. 
\end{proof}
\begin{tbd}
\begin{proof}
All balanced factorizations $(L,R)$ of $X=LR^\top$ are of the form $L=U\Sigma^{1/2}Q$ and $R=V\Sigma^{1/2}Q^\top$ where $X=U\Sigma V^\top$ is a thin SVD of $X$ and $Q\in O(r)$ is arbitrary~\cite[Lem.~2.33]{eitan_thesis}. 
Assume first that $\sigma_r(X) > 0$.
Since $(L, R)$ is balanced, we have $\|L^{\dagger}\|_{\mathrm{op}} = \|R^{\dagger}\|_{\mathrm{op}} = \sigma_r(X)^{-1/2}$.
Because $(L, R)$ is an $(\epsilon_1, \epsilon_2)$-approximate 2-critical point for $g$, it holds in particular that
\begin{align*}
   \sqrt{\|L^\top\nabla f(X)\|^2 + \|\nabla f(X)R\|^2} = \|\nabla g(L,R)\| \leq \epsilon_1.
\end{align*}
Since $\sigma_r(X)>0$, we have $\rank(X)=\rank(L)=\rank(R)=r$, so $\mathrm{col}(X)=\mathrm{col}(L)$ and $\mathrm{row}(X)=\mathrm{col}(R)$. Therefore,
\begin{align*}
    &\textrm{if } \nabla f(X) = U\begin{bmatrix} M_1 & M_2\\ M_3 & M_4\end{bmatrix}V^\top,\textrm{ then } LL^{\dagger}\nabla f(X) = \Proj_{\mathrm{col}(X)}\nabla f(X) = U\begin{bmatrix} M_1 & M_2 \\ 0 & 0\end{bmatrix}V^\top,\\ &\nabla f(X)RR^{\dagger} = \nabla f(X)\Proj_{\mathrm{row}(X)} = U\begin{bmatrix} M_1 & 0\\ M_3 & 0\end{bmatrix}V^\top, \textrm{ and } LL^{\dagger}\nabla f(X)RR^{\dagger} = U\begin{bmatrix} M_1 & 0\\ 0 & 0\end{bmatrix}V^\top.
\end{align*}
By~\eqref{eq:proj_Tx} we get:
\begin{align*}
    \|\Proj_{\T_X\Rmnlr}(-\nabla f(X))\| & = \|LL^\dagger\nabla f(X) + \nabla f(X)RR^{\dagger} - LL^\dagger\nabla f(X)RR^{\dagger}\| \\
            & \leq \|(L^{\dagger})^\top L^\top \nabla f(X) + \nabla f(X)RR^{\dagger}\| \\
            & \leq \sigma_r(X)^{-1/2}(\|L^\top\nabla f(X)\|+\|\nabla f(X)R\|) 
            \leq \sqrt{\frac{2}{\sigma_r(X)}} \epsilon_1,
\end{align*}
giving us one of the claimed bounds.

Next, assume only $\rank(X) \leq r$.
Since $(L, R)$ is balanced we can pick $w \in \RR^r$ such that $\|Lw\| = \|Rw\| = \sigma_r(X)^{1/2}$ and $\|w\| = 1$.
Let $u_1, v_1$ be top left and right singular vectors of $\nabla f(X)$, so $\|u_1\| = \|v_1\| = 1$ and $\nabla f(X)v_1 = \|\nabla f(X)\|_{\mathrm{op}} u_1$.
Set $\dot L = u_1w^\top$ and $\dot R = -v_1w^\top$.
Since $f$ has an $L_f$-Lipschitz continuous gradient, the Hessian $\nabla^2 f(X)$ has operator norm bounded by $L_f$.
Thus,
\begin{align*}
    \langle\nabla^2f(X)[\dot LR^\top\! +L\dot R^\top],\dot LR^\top\! + L\dot R^\top \rangle & \leq L_f\|\dot LR^\top\! +L\dot R^\top\|^2 \\
                   & \leq L_f\Big(\|u_1(Rw)^\top\| + \|Lwv_1^\top\|\Big)^2 
                   \leq 4L_f\sigma_r(X).
\end{align*}
We also have
\begin{align*}
    \langle\nabla f(X),\dot L\dot R^\top \rangle = -\|\nabla f(X)\|_{\mathrm{op}},
\end{align*}
and $\|\dot L\|=\|\dot R\|=1$.
By approximate 2-criticality and the above estimates, we get
\begin{align*}
    -2\epsilon_2 = -\epsilon_2(\|\dot L\|^2\!+\!\|\dot R\|^2) &\leq \langle \nabla^2g(L,R)[\dot L,\dot R],[\dot L,\dot R]\rangle\\
    &= \langle\nabla^2f(X)[\dot LR^\top\! + L\dot R^\top],\dot LR^\top\! + L\dot R^\top \rangle + 2\langle\nabla f(X),\dot L\dot R^\top \rangle\\
    &\leq 4L_f\sigma_r(X)-2\|\nabla f(X)\|_{\mathrm{op}},
\end{align*}
or equivalently,
\begin{align}
    \|\nabla f(X)\|_{\mathrm{op}} \leq \epsilon_2 + 2L_f\sigma_r(X).
\end{align}
This is the first claimed bound. 
Finally, $\|\Proj_{\T_X\Rmnlr}(-\nabla f(X))\|\leq \|\nabla f(X)\|\leq\sqrt{\rank(\nabla f(X))}\|\nabla f(X)\|_{\mathrm{op}}$ where the first inequality holds by Proposition~\ref{prop:normprojtocone}. 
\end{proof}
\end{tbd}

We can interpret Theorem~\ref{thm:stable_2implies1_LR} within the context of apocalypses.
Consider a balanced sequence $(L_k, R_k)_{k\geq1}$ and the corresponding sequence $(X_k)_{k\geq1}$ with $X_k = L_k^{}R_k^\top$. Suppose that $X_k \to X$ and $(X, (X_k)_{k\geq 1}, f)$ is an apocalypse on $\Rmnlr$ for some $f\colon\Rmn\to\reals$ with $L_f$-Lipschitz continuous gradient.
In particular, $\sigma_r(X_k) \to 0$ while $\|\nabla f(X_k)\| \to \|\nabla f(X)\| \neq 0$.
Inequality~\eqref{eq:opnormnablafXbound} in Theorem~\ref{thm:stable_2implies1_LR} then provides
\begin{align}
    \max(0, -\lambda_{\min}(\nabla^2 g(L_k, R_k))) \geq \|\nabla f(X_k)\|_\mathrm{op} - 2L_f\sigma_r(X_k),
\end{align}
whose right-hand side is (strictly) positive for all sufficiently large $k$.
This quantifies how second-order information on the lift enables us to escape apocalypses on $\Rmnlr$. For the apocalypse $(\bar{X}, (X_k)_{k\geq 1}, f)$ constructed in Section~\ref{subsec:apocalypse_on_mats}, the only balanced factorization for $\bar{X} = 0$ is $(L, R) = (0, 0)$, and $\lambda_{\min}(\nabla^2g(0,0)) = -1 = -\|\nabla f(\bar{X})\|_{\mathrm{op}}$.

\section{Conclusions and future work}

We defined the notion of apocalypses on nonsmooth constraint sets.
We constructed an explicit apocalypse on the variety of bounded rank matrices $\Rmnlr$ that is followed by the {\DPGD} algorithm of~\cite{schneider2015convergence}.
We characterized apocalypses in terms of limits of tangent cones, from which we concluded that Clarke-regular sets do not have apocalypses.
Focusing on the problem of computing stationary points on $\Rmnlr$, we gave an algorithm running on a lift of $\Rmnlr$ which uses second-order information about the cost function as well as a rebalancing map.
The iterates generated by the algorithm are guaranteed to have a 2-critical accumulation point if the cost function is three times continuously differentiable and has compact sublevel sets on $\Rmnlr$.
If the rebalancing map is sectional (e.g., Examples~\ref{ex:LR} and~\ref{ex:tensors}), we can view our algorithm as running on the original variety $\Rmnlr$ rather than on its lift.
We also quantified how second-order information allows us to escape apocalypses, by showing that the Hessian of the lifted cost function has a sufficiently negative eigenvalue near preimages of apocalypses.

This work suggests several directions for future investigation: 
\begin{enumerate}
    \item \textbf{Stability of apocalypses:} In our example of an apocalypse, the set of initializations from which the {\DPGD} algorithm follows that apocalypse has measure zero. Is this always true? 
    
    A partial such result was shown in~\cite{hou2021asymptotic} for cost functions of the form $f(X) = \tfrac{1}{2}\|X-X_0\|^2$ restricted to positive semidefinite matrices of bounded rank.
    
    \item \textbf{Effect of apocalypses on algorithms:} Can apocalypses slow down local algorithms in their neighborhood, and if so, can randomization help? This is the case for saddle points~\cite{NIPS2017_f79921bb}.
    
    \item \textbf{Continuity of measures of approximate stationarity:} The definition of apocalypses depends on our notion of approximate stationarity, and their existence is closely related to the discontinuity of our measure of approximate stationarity, namely the function $x\mapsto \|\Proj_{\T_x\calX}(-\nabla f(x))\|$. Is it possible to construct \emph{tractable} and \emph{continuous} measures of approximate stationarity on nonsmooth sets, perhaps without relying on projections to tangent cones? There can be no apocalypses with respect to such a measure. We emphasize tractability here since the distance to the set of stationary points of $f$ is an intractable example of a continuous stationarity measure.
    
    \item \textbf{First-order algorithms on non-smooth sets:} In this paper we computed stationary points on $\Rmnlr$ by using second-order information about the cost function on a lift. The ingredients we used to achieve this, namely, Algorithm~\ref{algo:TRbalanced}, rebalancing maps, and smooth lifts, extend to optimization problems~\eqref{eq:P} over more general constraint sets $\calX$, as do apocalypses. It is therefore natural to ask: Is there an algorithm running directly on a more general class of non-smooth sets $\calX$ that only uses first-order information about the cost function, and which is guaranteed to converge to a stationary point?
    
    In the important special case of $\calX=\Rmnlr$, this question was answered positively in~\cite{olikier2022apocalypse}, as discussed in Section~\ref{sec:intro}.
    
    \item \textbf{Non-Clarke regular sets without apocalypses:} Can the implications in Corollary~\ref{cor:apocalypse_suff_cond} be reversed, or are there non-Clarke regular sets that do not have apocalypses?

    \item \textbf{Convergence without apocalypses:}
    For convex optimization, {\DPGD} with appropriate step sizes finds a global optimum (this can be shown by a straightforward modification of the standard analysis for projected gradient descent.)
    More generally, is it the case that {\DPGD} converges to stationary points on apocalypse-free nonsmooth sets, or are there other obstacles?
\end{enumerate}

\begin{acknowledgements}
Part of this work was conducted while all three authors were affiliated with the mathematics department and PACM at Princeton University.
\end{acknowledgements}

\bibliographystyle{abbrv}
\bibliography{bibl_lifts}

\clearpage

\appendix

\section{Facts about cones}
\label{apdx:cones}
Let $\calE$ be a Euclidean space equipped with an inner product $\inner{\cdot}{\cdot}$. 
A cone is a set $K \subseteq \calE$ with the following property: $u \in K \implies \alpha u \in K \ \forall \alpha > 0$.
The polar cone $K^\circ$ is defined in Definition~\ref{def:polardual}.
We begin by stating several standard results about convex cones and their polars:
\begin{lemma}[{\cite[Prop.~4.5]{deutsch2012best}}] \label{lem:coneproperties}
	Let $K$ be a cone in $\calE$.
	Then, $K^\circ$ is a closed, convex cone in $\calE$,
	and $K^{\circ\circ} = (K^\circ)^\circ$ is a closed, convex cone equal to $\overline{\conv}(K)$ (the closure of the convex hull of $K$).
	In particular, if $K$ is closed and convex, then $K^{\circ\circ} = K$.
	
	Furthermore, the following properties hold:
	\begin{itemize}
	    \item Let $K_1,K_2\subset\calE$ be cones. If $K_1 \subseteq K_2$, then $K_2^\circ \subseteq K_1^\circ$.
	\item The polar of a subspace $V \subseteq \calE$ is its orthogonal complement: $V^\circ = V^\perp$.
	\item $K \subseteq K^{\circ\circ}$.
\end{itemize}

\end{lemma}


\begin{lemma}[{\cite[Prop.~6.27]{bauschke2011convex}}] \label{lem:coneprojection}
	Let $K$ be a closed, convex cone in $\calE$.
	Let $v \in \calE$ be arbitrary.
	The projection of $v$ to $K$ exists and is unique.
	This projection is the unique vector $u \in \calE$ satisfying
	\begin{align*}
		u \in K, && v-u \in K^\circ && \textrm{ and } && \inner{u}{v-u} = 0.
	\end{align*}
\end{lemma}

\begin{corollary}[{\cite[Thm.~6.29]{bauschke2011convex}}]\label{cor:coneprojectionsplit}
    Let $K$ be a closed, convex cone in $\calE$.
	Given $v \in \calE$ arbitrary, let $u = \Proj_K(v)$ and $w = \Proj_{K^\circ}(v)$.
	Then, 
	\begin{align*}
	    v = u + w, && \inner{u}{w} = 0 && \textrm{ and } && \|v\|^2 = \dist(v,K)^2 + \dist(v,K^\circ)^2.
	\end{align*}
	
	In particular, $K+K^\circ = \calE$.
\end{corollary}

\begin{corollary} \label{cor:mu1mu2equalconvex}
	If $K$ is a closed, convex cone in $\calE$, then for all $v \in \calE$ we have:
    $$\|\Proj_K(v)\| = \dist(v, K^\circ).$$
\end{corollary}
\begin{proof}
By Corollary~\ref{cor:coneprojectionsplit}, $\|\Proj_K(v)\| = \|v - \Proj_{K^{\circ}}(v)\| = \dist(v,K^\circ)$.
\end{proof}

\begin{remark}\label{rmk:everything_breaks_if_noncvx}
If $K$ is not convex, the last three results can break. 
To see this, consider $K$ composed of three rays in the plane starting from the origin, with two of them forming an angle $4\alpha \in (0, \pi)$ and the other two angles being equal to $\pi - 2\alpha$.
Notice that $K$ is not convex, and that $K^\circ = \{0\}$.
In particular, $K + K^\circ \neq \calE$.
Let $v$ be some arbitrary vector in $\calE$.
Of course, $w = \Proj_{K^\circ}(v) = 0$.
Thus, if $v \notin K$, we see that $v \neq u + w$ with $u = \Proj_K(v)$.
Now take $v$ to bisect one of the angles $\pi - 2\alpha$.
Then, $\dist(v,K^\circ) = \|v - w\| = \|v\|$, yet $\|\Proj_K(v)\| = \|u\| = \cos(\pi/2 - \alpha) \|v\|$.
We see that $\|\Proj_K(v)\| / \dist(v,K^\circ)$ can be arbitrarily close to zero.

Even the property that $\inner{u}{w} = 0$ can fail.
Indeed, for some $\epsilon > 0$, let $K$ be the cone in $\reals^3$ consisting of the three rays generated by $x_1 = (1, 1, 0)\transpose$, $x_2 = (-1, 1, 0)\transpose$ and $x_3 = (0, 1, -\epsilon)\transpose$, and let $v = (0, 1, \epsilon)\transpose$.
Then, $q = \Proj_{K^{\circ\circ}}(v) = (0, 1, 0)\transpose$ (this is geometrically clear, and can be verified formally with Lemma~\ref{lem:coneprojection}).
We know from Corollary~\ref{cor:coneprojectionsplit} that $v = q + w$ where $w = \Proj_{K^\circ}(v)$, so that $w = v - q = (0, 0, \epsilon)\transpose$.
Notice that $w$ is orthogonal to $x_1$ and $x_2$ but not to $x_3$.
Yet, for $\epsilon$ small enough the projection $u = \Proj_K(v)$ is unique, equal to a nonzero multiple of $x_3$, so that $\inner{u}{w} \neq 0$.
To verify this last statement, simply consider the projection of $v$ to each of the rays generated by $x_1, x_2, x_3$.
Clearly, $y_i \triangleq \Proj_{\mathrm{cone}(x_i)}(v) = \max\!\left(0, \frac{\inner{x_i}{v}}{\inner{x_i}{x_i}}\right) x_i$.
Thus, $y_1 = (1/2, 1/2, 0)\transpose$, $y_2 = (-1/2, 1/2, 0)\transpose$ and $y_3 = \frac{1-\epsilon^2}{1+\epsilon^2} x_3$ (if $\epsilon \leq 1$).
The distances to these projections are $\|v - y_1\| = \|v - y_2\| = \sqrt{1/2 + \epsilon^2}$ and $\|v - y_3\| = \frac{2\epsilon}{\sqrt{1+\epsilon^2}}$.
It follows that $u$ is a nonzero multiple of $x_3$ for all $0 < \epsilon < \sqrt{5 - \sqrt{17}}/2$.
\begin{journal}
See the arXiv version of this paper for counterexamples~\cite[Rmk.~A.5]{levin2021finding}.
\end{journal}
\begin{tbd}
To see this, consider $K$ composed of three rays in the plane starting from the origin, with two of them forming an angle $4\alpha \in (0, \pi)$ and the other two angles being equal to $\pi - 2\alpha$.
Notice that $K$ is not convex, and that $K^\circ = \{0\}$.
In particular, $K + K^\circ \neq \calE$.
Let $v$ be some arbitrary vector in $\calE$.
Of course, $w = \Proj_{K^\circ}(v) = 0$.
Thus, if $v \notin K$, we see that $v \neq u + w$ with $u = \Proj_K(v)$.
Now take $v$ to bisect one of the angles $\pi - 2\alpha$.
Then, $\dist(v,K^\circ) = \|v - w\| = \|v\|$, yet $\|\Proj_K(v)\| = \|u\| = \cos(\pi/2 - \alpha) \|v\|$.
We see that $\|\Proj_K(v)\| / \dist(v,K^\circ)$ can be arbitrarily close to zero.

Even the property that $\inner{u}{w} = 0$ can fail.
Indeed, for some $\epsilon > 0$, let $K$ be the cone in $\reals^3$ consisting of the three rays generated by $x_1 = (1, 1, 0)\transpose$, $x_2 = (-1, 1, 0)\transpose$ and $x_3 = (0, 1, -\epsilon)\transpose$, and let $v = (0, 1, \epsilon)\transpose$.
Then, $q = \Proj_{K^{\circ\circ}}(v) = (0, 1, 0)\transpose$ (this is geometrically clear, and can be verified formally with Lemma~\ref{lem:coneprojection}).
We know from Corollary~\ref{cor:coneprojectionsplit} that $v = q + w$ where $w = \Proj_{K^\circ}(v)$, so that $w = v - q = (0, 0, \epsilon)\transpose$.
Notice that $w$ is orthogonal to $x_1$ and $x_2$ but not to $x_3$.
Yet, for $\epsilon$ small enough the projection $u = \Proj_K(v)$ is unique, equal to a nonzero multiple of $x_3$, so that $\inner{u}{w} \neq 0$.
To verify this last statement, simply consider the projection of $v$ to each of the rays generated by $x_1, x_2, x_3$.
Clearly, $y_i \triangleq \Proj_{\mathrm{cone}(x_i)}(v) = \max\!\left(0, \frac{\inner{x_i}{v}}{\inner{x_i}{x_i}}\right) x_i$.
Thus, $y_1 = (1/2, 1/2, 0)\transpose$, $y_2 = (-1/2, 1/2, 0)\transpose$ and $y_3 = \frac{1-\epsilon^2}{1+\epsilon^2} x_3$ (if $\epsilon \leq 1$).
The distances to these projections are $\|v - y_1\| = \|v - y_2\| = \sqrt{1/2 + \epsilon^2}$ and $\|v - y_3\| = \frac{2\epsilon}{\sqrt{1+\epsilon^2}}$.
It follows that $u$ is a nonzero multiple of $x_3$ for all $0 < \epsilon < \sqrt{5 - \sqrt{17}}/2$.
\end{tbd}
\end{remark}

Despite that remark, we can secure certain properties in the nonconvex case. 
\begin{proposition} \label{prop:normprojtocone}
	Let $K$ be a closed cone (not necessarily convex) in $\calE$.
	Let $v \in \calE$ be arbitrary and let $u$ be a projection of $v$ to $K$ (note that $u$ exists but is not necessarily unique).
	We have
	\begin{align*}
		\|u\| & = \max\!\left(0, \max_{z \in K, \|z\| = 1} \inner{v}{z} \right) = \sqrt{\inner{v}{u}}.
	\end{align*}
	In particular:
	\begin{itemize}
	    \item all projections of $v$ to $K$ have the same norm,
	    \item $\Proj_K(v) = 0 \iff v \in K^\circ$,
	    \item $v\mapsto \|\Proj_K(v)\|$ is 1-Lipschitz.
	\end{itemize}
\end{proposition}
\begin{proof}
	There exists $z \in K$ with $\|z\| = 1$ and $\alpha \geq 0$ such that $u = \alpha z$.
	By definition of $u$, it holds that $z$ and $\alpha$ jointly minimize
	\begin{align*}
		\|v - \alpha z\|^2 = \alpha^2 - 2\alpha \inner{v}{z} + \|v\|^2.
	\end{align*}
	This is a convex quadratic in $\alpha$.
	For any given $z$, the optimal $\alpha \geq 0$ is given by
	\begin{align*}
		\alpha = \max(0, \inner{v}{z}).
	\end{align*}
	Thus, $z$ is the vector in $K$ with $\|z\| = 1$ which minimizes
	\begin{align*}
		\alpha \left( \alpha - 2 \inner{v}{z}\right) = g(\inner{v}{z}) && \textrm{ with } && g(t) = \max(0, t) \left( \max(0, t) - 2t \right).
	\end{align*}
	The function $g \colon \reals \to \reals$ is decreasing.
	Therefore, $z$ is the vector in $K$ with $\|z\| = 1$ which maximizes $\inner{v}{z}$.
	We conclude that
	\begin{align*}
		\|u\| = \|\alpha y\| = \alpha = \max\!\left( 0, \max_{z \in K, \|z\| = 1} \inner{v}{z} \right).
	\end{align*}
	Furthermore, either $\|u\|>0$ in which case the above shows 
	\begin{align*}
	    \|u\| = \max_{z\in K, \|z\|=1}\inner{v}{z} = \inner{v}{\frac{u}{\|u\|}},
	\end{align*}
	hence $\inner{v}{u}=\|u\|^2$, or $\|u\|=0$ in which case $\inner{v}{u}=\|u\|^2=0$. 
	
	Finally, if $v,v'\in\calE$ and $z\in K$ satisfies $\|z\|=1$, then $\langle v, z\rangle = \langle v', z\rangle + \langle v-v',z\rangle$.
	Taking max over such $z$ gives $\|\Proj_K(v)\|\leq \|\Proj_K(v')\| + \|v-v'\|$, and reversing the roles of $v$ and $v'$ shows that $v\mapsto \|\Proj_K(v)\|$ is 1-Lipschitz.
\end{proof}

Next, we show that $\Proj_{\T_x\calX}(-\nabla f(x))$ only depends on the values of $f$ on $\calX$, not on its particular extension to $\calE$. This desirable property means that algorithms designed to be insensitive to the (somewhat arbitrary) values of $f$ outside of the feasible set $\calX$ can use $\Proj_{\T_x\calX}(-\nabla f(x))$ to construct update rules.
This is notably the case for the {\DPGD} algorithm of~\cite{schneider2015convergence}.
\begin{lemma}\label{lem:ProjKvinvariant}
	Let $K$ be a closed cone (not necessarily convex) in $\calE$.
	Let $v \in \calE$ be arbitrary.
	Split $v = u + w$ with $u$ in the linear span of $K$ and $w$ orthogonal to that linear space.
	Then, $\Proj_K(v)$ (the \emph{set} of projections of $v$ to $K$) is identical to $\Proj_K(u)$.
	In particular, $\|\Proj_K(v)\| = \|\Proj_K(u)\|$ and $\dist(v, K^\circ) = \dist(u, K^\circ)$.
\end{lemma}
\begin{proof}
	Let $z \in K$ be arbitrary.
	Then, $\|z - v\|^2 = \|(z-u) - w\|^2 = \|z-u\|^2 + \|w\|^2$.
	Therefore, the set of elements $z \in K$ which minimize $\|z - v\|$ is exactly the same as the set of elements $z \in K$ which minimize $\|z-u\|$, that is, $\Proj_K(v) = \Proj_K(u)$.
	For the second claim, Lemma~\ref{lem:coneproperties} implies that $K^{\circ\circ}$ is a closed, convex cone satisfying
	$(K^{\circ\circ})^\circ = (K^\circ)^{\circ\circ} = K^\circ$. Corollary~\ref{cor:mu1mu2equalconvex} then implies $\dist(v, K^\circ) = \|\Proj_{K^{\circ\circ}}(v)\|$ so we can apply the first claim.
\end{proof}
Lemma~\ref{lem:ProjKvinvariant} shows that $\Proj_{\T_x\calX}(-\nabla f(x))$ only depends on $\Proj_{\mathrm{span}(\T_x\calX)}(-\nabla f(x))$. We proceed to show that the latter projection only depends on $f|_{\calX}$.
\begin{lemma} \label{lem:nablabarfinvariance}
	Let $f \colon \calE \to \reals$ be differentiable on a neighborhood of $\calX \subseteq \calE$.
	While $\nabla f(x)$ may depend on $f$ through its values outside of $\calX$, the orthogonal projection of $\nabla f(x)$ to the linear span of the tangent cone $\T_x\calX$ is fully determined by $f|_{\calX}$.
\end{lemma}
\begin{proof}
	Let $v \in \T_x\calX$ be arbitrary.
	By Definition~\ref{def:tangentcone}, there exist sequences $(x_k)$ in $\calX$ and $(\tau_k)$ in $\reals$ with $\tau_k > 0$ for all $k$ and $\lim_{k \to \infty} \tau_k = 0$ such that $v = \lim_{k \to \infty} \frac{x_k - x}{\tau_k}$.
	Since $f$ is differentiable at $x$ it follows that for all $k$ we have
	\begin{align*}
		\frac{f(x_k) - f(x)}{\tau_k} = \frac{\D f(x)[x_k - x] + o(\|x_k - x\|)}{\tau_k}.
	\end{align*}
	Considering the right-hand side, we see that the limit for $k \to \infty$ exists and is equal to $\D f(x)[v]$.
	Therefore,
	\begin{align*}
		\innersmall{\nabla f(x)}{v} = \D f(x)[v] = \lim_{k \to \infty} \frac{f(x_k) - f(x)}{\tau_k}.
	\end{align*}
	This confirms that the inner product of $\nabla f(x)$ with any element of the tangent cone $\T_x\calX$ is fully determined by $f|_{\calX}$.
	By linearity, this extends to all $v$ in the linear span of $\T_x\calX$.
\end{proof}

Our definition of apocalypses in Definition~\ref{def:apocalypses} depends on the notion of approximate stationarity we use. Instead of relaxing the characterization of stationarity in Proposition~\ref{prop:equiv_defns_of_1crit}(c) as we did in Definition~\ref{def:criticalPapprox}, we could relax Proposition~\ref{prop:equiv_defns_of_1crit}(b) by requiring $\dist(-\nabla f(x),(\T_x\calX)^{\circ})\leq\epsilon$ at approximate stationary points. 
We proceed to relate these two measures of approximate stationarity and show that Definition~\ref{def:criticalPapprox} is more natural and introduces at most as many apocalypses as the alternative.

\begin{theorem} \label{thm:invarianceprojsummary}
	With $v = -\nabla f(x)$ and $K = \T_x\calX$,
	let $\mu_1 = \|\Proj_K(v)\|$ and $\mu_2 = \dist(v, K^\circ)$ denote two measures of approximate stationarity for $x \in \calX$. Then:
	\begin{enumerate}[(a)]
	    \item The set $\Proj_K(v)$ and the measures $\mu_1, \mu_2$ depend on $f$ only through its restriction $f|_{\calX}$.
	    \item $\mu_1$ is the best linear rate of decrease one can achieve in $f$ by moving away from $x$ along a direction in the tangent cone $\T_x\calX$, while $\mu_2$ is the best such rate one can achieve by moving away from $x$ along a direction in the (possibly larger) cone $\overline{\conv}(\T_x\calX)$.
	    \item It always holds that $\mu_1 \leq \mu_2$ but the inequality may be strict.
	    \item $\mu_1 = 0 \iff \mu_2 = 0$.
	    \item If $K$ is convex (in particular, if it is a linear subspace), then $\mu_1 = \mu_2$.
	\end{enumerate}\end{theorem}
\begin{proof}
    Part (a) follows from Lemmas~\ref{lem:ProjKvinvariant}-\ref{lem:nablabarfinvariance}. Part (b) follows from Proposition~\ref{prop:normprojtocone}, Corollary~\ref{cor:mu1mu2equalconvex}, and Lemma~\ref{lem:coneproperties}. The inequality in part (c) follows directly from (b), and it can be strict by Remark~\ref{rmk:everything_breaks_if_noncvx}.
	Part (d) follows from Proposition~\ref{prop:normprojtocone}. Part (e) follows from Corollary~\ref{cor:mu1mu2equalconvex}.
\end{proof}
Theorem~\ref{thm:invarianceprojsummary}(b) suggests that $\mu_2$ is less natural as a notion of approximate stationarity. Furthermore, since $\mu_1$ and $\mu_2$ agree when $\T_x\calX$ is a linear space and $\mu_1=0\iff\mu_2=0$, the apocalypse we construct in Section~\ref{subsec:apocalypse_on_mats} also ``fools'' $\mu_2$.
\begin{tbd}
Our definition of apocalypses in Definition~\ref{def:apocalypses} depends on the notion of approximate stationarity we use. Instead of relaxing the characterization of stationarity in Proposition~\ref{prop:equiv_defns_of_1crit}(c) as we did in Definition~\ref{def:criticalPapprox}, we could relax Proposition~\ref{prop:equiv_defns_of_1crit}(b) by requiring $\dist(-\nabla f(x),(\T_x\calX)^{\circ})\leq\epsilon$ at approximate stationary points. 
We proceed to relate these two measures of approximate stationarity and show that Definition~\ref{def:criticalPapprox} is more natural and introduces at most as many apocalypses as the alternative.

\begin{theorem} \label{thm:invarianceprojsummary}
	With $v = -\nabla f(x)$ and $K = \T_x\calX$,
	let $\mu_1 = \|\Proj_K(v)\|$ and $\mu_2 = \dist(v, K^\circ)$ denote two measures of approximate stationarity for $x \in \calX$. Then:
	\begin{enumerate}[(a)]
	    \item The set $\Proj_K(v)$ and the measures $\mu_1, \mu_2$ depend on $f$ only through its restriction $f|_{\calX}$.
	    \item $\mu_1$ is the best linear rate of decrease one can achieve in $f$ by moving away from $x$ along a direction in the tangent cone $\T_x\calX$, while $\mu_2$ is the best such rate one can achieve by moving away from $x$ along a direction in the (possibly larger) cone $\overline{\conv}(\T_x\calX)$.
	    \item It always holds that $\mu_1 \leq \mu_2$ but the inequality may be strict.
	    \item $\mu_1 = 0 \iff \mu_2 = 0$.
	    \item If $K$ is convex (in particular, if it is a linear subspace), then $\mu_1 = \mu_2$.
	\end{enumerate}\end{theorem}
\begin{proof}
    Part (a) follows from Lemmas~\ref{lem:ProjKvinvariant}-\ref{lem:nablabarfinvariance}. Part (b) follows from Proposition~\ref{prop:normprojtocone}, Corollary~\ref{cor:mu1mu2equalconvex}, and Lemma~\ref{lem:coneproperties}. The inequality in part (c) follows directly from (b), and it can be strict by Remark~\ref{rmk:everything_breaks_if_noncvx}.
	Part (d) follows from Proposition~\ref{prop:normprojtocone}. Part (e) follows from Corollary~\ref{cor:mu1mu2equalconvex}.
\end{proof}
Theorem~\ref{thm:invarianceprojsummary}(b) suggests that $\mu_2$ is less natural as a notion of approximate stationarity. Furthermore, since $\mu_1$ and $\mu_2$ agree when $\T_x\calX$ is a linear space and $\mu_1=0\iff\mu_2=0$, the apocalypse we construct in Section~\ref{subsec:apocalypse_on_mats} also ``fools'' $\mu_2$.
\end{tbd}

\section{Proof of Theorem~\ref{thm:apocalypse_char}}
\label{apdx:pf_of_apocs_char}

\begin{lemma} \label{lem:polar_of_limsup}
    Suppose $(x_k)_{k\geq 1}\subseteq\calX$ is a sequence converging to $x\in\calX$. Then
    \begin{equation} 
        \left(\limsup_{k\to\infty}\T_{x_k}\calX\right)^{\circ} = \{v\in\calE:\lim_{k\to\infty}\|\Proj_{\T_{x_k}}(v)\| = 0\}.
    \end{equation}
\end{lemma}
\begin{proof}
Suppose $v\in (\limsup_k\T_{x_k}\calX)^{\circ}$. Let $u_k = \Proj_{\T_{x_k}\calX}(v)$. If $u_k\not\to 0$, we can find a subsequence $(u_{k_i})_{i\geq 1}$ such that $\|u_{k_i}\|\geq r>0$ for all $i$ and such that $\frac{u_{k_i}}{\|u_{k_i}\|}$ converges to some $\hat u$. Because the $\T_{x_k}\calX$ are cones, we have $\frac{u_{k_i}}{\|u_{k_i}\|}\in \T_{x_{k_i}}\calX$ for all $i$, so $\hat u\in \limsup_k\T_{x_k}\calX$ by Definition~\ref{def:lims_of_sets}. Moreover, we have 
$$\langle v,\hat u\rangle = \lim_i\frac{\langle v, u_{k_i}\rangle}{\|u_{k_i}\|} = \lim_i\|u_{k_i}\|\geq r>0,$$ 
since $\langle v,u_{k_i}\rangle=\|u_{k_i}\|^2$ by Proposition~\ref{prop:normprojtocone}. This contradicts $v\in (\limsup_k\T_{x_k}\calX)^{\circ}$, so $u_k\to 0$. 

Conversely, suppose $\lim_k\|\Proj_{\T_{x_k}\calX}(v)\|=0$. For any $u\in \limsup_k\T_{x_k}\calX$, write $u=\lim_iu_{k_i}$ where $u_{k_i}\in \T_{x_{k_i}}\calX$. If $u=0$ we have $\langle v,u\rangle=0$. Otherwise, by passing to a subsequence again we may assume $u_{k_i}\neq 0$ for all $i$. We then have
$$ \langle v, u\rangle = \lim_i\langle v, u_{k_i}\rangle = \lim_i\|u_{k_i}\|\inner{v}{\frac{u_{k_i}}{\|u_{k_i}\|}}\leq \limsup_i\|u_{k_i}\|\cdot\|\text{Proj}_{\T_{x_{k_i}}}(v)\| = 0,$$ 
where for the last inequality we used Proposition~\ref{prop:normprojtocone} again. This shows that $v\in (\limsup_k\T_{x_k}\calX)^{\circ}$.
\end{proof}

We are now ready to prove the theorem:
\begin{proof}[Proof of Theorem~\ref{thm:apocalypse_char}]
Let $(x, (x_k)_{k\geq 1},f)$ be an apocalypse.
By Proposition~\ref{prop:normprojtocone}, 
\begin{align*}
    \|\Proj_{\T_{x_k}\calX}(-\nabla f(x))\| \leq \|\Proj_{\T_{x_k}\calX}(-\nabla f(x_k))\| + \|\nabla f(x)-\nabla f(x_k)\| \xrightarrow{k\to\infty} 0.
\end{align*}
By Lemma~\ref{lem:polar_of_limsup}, we then have $-\nabla f(x) \in (\limsup_k\T_{x_k}\calX)^{\circ}$. On the other hand, because $\Proj_{\T_x\calX}(-\nabla f(x))\neq 0$ we have $-\nabla f(x)\notin (\T_x\calX)^{\circ}$ by Proposition~\ref{prop:normprojtocone}, hence $-\nabla f(x)\in (\limsup_k\T_{x_k}\calX)^{\circ}\setminus (\T_x\calX)^{\circ}$. Conversely, if $-\nabla f(x)\in (\limsup_i\T_{x_k}\calX)^{\circ}\setminus(\T_x\calX)^{\circ}$ then Proposition~\ref{prop:normprojtocone} and~\ref{lem:polar_of_limsup} give
\begin{align*}
    \|\Proj_{\T_{x_k}\calX}(-\nabla f(x_k))\|\leq \|\Proj_{\T_{x_k}}(-\nabla f(x))\| + \|\nabla f(x)-\nabla f(x_k)\|\xrightarrow{k\to\infty}0,
\end{align*}
while $\Proj_{\T_x\calX}(-\nabla f(x))\neq0$, hence $(x,(x_k)_{k\geq 1},f)$ is an apocalypse.

For the second claim, 
if $x$ is apocalyptic, then by definition there exists a sequence $(x_k)_{k\geq 1}\subseteq\calX$ converging to $x$ and a $\mc C^1$ function $f$ such that $(x,(x_k)_{k\geq 1},f)$ is an apocalypse. 
The first part of the proof shows that $-\nabla f(x)\in \left(\limsup_k\T_{x_k}\calX\right)^{\circ}\setminus(\T_x\calX)^{\circ}$.
Conversely, suppose $(\limsup_k\T_{x_k}\calX)^{\circ}\not\subseteq (\T_x\calX)^{\circ}$ for some sequence $(x_k)_{k\geq 1}$ converging to $x$, and let $-v\in \left(\limsup_k \T_{x_k}\calX\right)^{\circ}\setminus(\T_x\calX)^{\circ}$. 
Define $f(x)=\langle x,v\rangle$, and note that $\Proj_{\T_{x_k}}(-\nabla f(x_k)) = \Proj_{\T_{x_k}}(-v)\to 0$ by Lemma~\ref{lem:polar_of_limsup}. 
However, $\Proj_{\T_x\calX}(-\nabla f(x)) = \Proj_{\T_x\calX}(-v) \neq 0$ because $-v\notin (\T_x\calX)^{\circ}$. 
Therefore, the triplet $(x, (x_k)_{k\geq 1}, f)$ is an apocalypse on $\calX$, so $x$ is apocalyptic.

Finally, by Lemma~\ref{lem:coneproperties}, $\left(\limsup_k\T_{x_k}\calX\right)^{\circ}\not\subseteq (\T_x\calX)^{\circ}$ if and only if \\ $\overline{\mathrm{conv}}(\T_x\calX)\not\subseteq \overline{\tt{conv}}(\limsup_k\T_{x_k}\calX)$, which is equivalent to
$\T_x\calX\not\subseteq \overline{\tt{conv}}(\limsup_k\T_{x_k}\calX)$.
\end{proof} 
\section{Analysis of Algorithm~\ref{algo:TRbalanced}} \label{appdx:hookedRTR}
We analyze Algorithm~\ref{algo:TRbalanced} by following in the steps of Curtis et al.~\cite{curtis2018concise} and checking that the modifications we have made do not impact the key claims.
The first lemma below guarantees each step provides non-trivial decrease in the model value.
\begin{lemma} \label{lem:optimlemma0}
	For all $k$, the step $s_k$ chosen in Step~\ref{step:sk} of Algorithm~\ref{algo:TRbalanced} satisfies
	\begin{align}
		m_k(0) - m_k(s_k) & \geq \begin{cases}
			\frac{1}{2} \min\!\left( \frac{1}{1+\|\nabla^2 \hat g_k(0)\|}, \gamma_k \right) \|\nabla g(y_k)\|^2 & \textrm{ if } k \in \calK_1, \\
			\frac{1}{2} \gamma_k^2 |(\lambda_k)_-|^3 & \textrm{ if } k \in \calK_2.
		\end{cases}
		\label{eq:optimdecrease}
	\end{align}
\end{lemma}
\begin{proof}
	This is exactly~\cite[Lem.~2.2]{curtis2018concise}.
    The model $m_k$
    is quadratic on a linear space ($\T_{y_k}\calM$): it suffices to analyze the Cauchy step explicitly, then to note that $m_k(s_k) \leq m_k(s_k^c)$, using our explicit knowledge of $\delta_k$ for the two types of steps ($\calK_1, \calK_2$).
\end{proof}

Moreover, an iterate $y_k$ is 2-critical if and only if $\nabla g(y_k) = 0$ and $\lambda_k \geq 0$.
\begin{lemma} \label{lem:optim2criticalcharact}
	For any retraction $\Retr$, a point $y \in \calM$ is 2-critical for $\min_{y\in\calM} g(y)$ if and only if $\nabla g(y) = 0$ and $\lambdamin(\nabla^2 (g \circ \Retr_y)(0)) \geq 0$.
\end{lemma}
\begin{proof}
	This holds because $\nabla g(y) = 0$ implies $\nabla^2(g \circ \Retr_y)(0) = \nabla^2 g(y)$, even if the retraction is not second order~\cite[Prop.~5.5.6]{absil_book}.
\end{proof}

In the computation of $\rho_k$ in Step~\ref{step:rhok}, we handle the special (if unlikely) case of a vanishing denominator.
The following lemma justifies the proposed rule, in that it can be used to check that all claims made below remain valid in that corner case.
\begin{lemma}
    Following
    Step~\ref{step:sk} in Algorithm~\ref{algo:TRbalanced}, the following are equivalent:
	\begin{align*}
		(a) \ s_k = 0, && (b) \ m_k(0) - m_k(s_k) = 0, && (c) \ y_k \textrm{ is 2-critical}.
	\end{align*}
\end{lemma}
\begin{proof}
	It is clear that $(a)$ implies $(b)$.
	We also see that $(c)$ implies $(a)$, because when $\lambda_k \geq 0$ and $\|\nabla g(y_k)\| = 0$ we necessarily have $\delta_k = 0$, so that $s_k = 0$ is the only admissible trial step.
	It remains to show that $(b)$ implies $(c)$.
	Assuming $(b)$, notice that $k$ is in $\calK_1$; indeed, $k \in \calK_2$ would imply $\lambda_k < 0$, in which case we would know from Lemma~\ref{lem:optimlemma0} that $m_k(0) - m_k(s_k) > 0$.
	Knowing that $k$ is in $\calK_1$, we use Lemma~\ref{lem:optimlemma0} again to infer from $(b)$ that $\|\nabla g(y_k)\| = 0$. 
	Then, combining the facts $\|\nabla g(y_k)\| = 0$ and $k \notin \calK_2$, we deduce that $\lambda_k \geq 0$ or $|(\lambda_k)_-|^3 \leq 0$.
	In all cases, it follows that $\lambda_k \geq 0$.
	Conclude with Lemma~\ref{lem:optim2criticalcharact}.
\end{proof}

Under the Lipschitz-type Assumption~\ref{assu:upstairs}, we further show:
\begin{lemma} \label{lem:optimlemma1}
	With Assumption~\ref{assu:upstairs}, for all $k$, the step $s_k$ from Step~\ref{step:sk} satisfies
	\begin{align}
		|g(\Retr_{y_k}(s_k)) - m_k(s_k)| & \leq \begin{cases}
			L_1 \gamma_k^2 \|\nabla g(y_k)\|^2 & \textrm{ if } k \in \calK_1, \\
			\frac{1}{6} L_2 \gamma_k^3 |(\lambda_k)_-|^3 & \textrm{ if } k \in \calK_2.
		\end{cases}
	\end{align}
\end{lemma}
\begin{proof}
	This is~\cite[Lem.~2.3]{curtis2018concise}.
	We use the Lipschitz-type inequalities provided by Assumption~\ref{assu:upstairs} and our explicit knowledge of $\delta_k$ for the two types of steps ($\calK_1, \calK_2$).
\end{proof}
Recall the definitions of $\gammamin$~\eqref{eq:gammamin} and $\kappamin$~\eqref{eq:kappamin}---we use them going forward.
\begin{lemma} \label{lem:optimgammamin}
	Under Assumption~\ref{assu:upstairs}, we have $\gamma_k \geq \gammamin$ for all $k$.
\end{lemma}
\begin{proof}
	This is~\cite[Lem.~2.4]{curtis2018concise},
    by induction relying on
    Lemmas~\ref{lem:optimlemma0} and~\ref{lem:optimlemma1}.
\end{proof}

Under the additional Assumption~\ref{assu:Phidecrease} on the rebalancing map, we have:
\begin{lemma} \label{lem:optimlemma2}
	Under Assumption~\ref{assu:upstairs}, for all $k$, the trial step $s_k$ satisfies
	\begin{align}
		m_k(0) - m_k(s_k) \geq \begin{cases}
			\frac{1}{2} \gammamin \|\nabla g(y_k)\|^2 & \textrm{ if } k \in \calK_1, \\
			\frac{1}{2} \gammamin^2 |(\lambda_k)_-|^3 & \textrm{ if } k \in \calK_2.
		\end{cases}
	\end{align}
	If $\Phi$ also satisfies Assumption~\ref{assu:Phidecrease}, then for all $k \in \calS = \{ k : \rho_k \geq \eta \}$ it holds that
	\begin{align}
		g(y_k) - g(y_{k+1}) \geq \begin{cases}
			\kappamin \|\nabla g(y_k)\|^2 & \textrm{ if } k \in \calK_1 \cap \calS, \\
			\kappamin |(\lambda_k)_-|^3   & \textrm{ if } k \in \calK_2 \cap \calS.
		\end{cases}
	\end{align}
	($\calS$ indexes the successful iterations.)
\end{lemma}
\begin{proof}
	This is~\cite[Lem.~2.5]{curtis2018concise}: it follows from Lemmas~\ref{lem:optimlemma0}, \ref{lem:optimlemma1} and~\ref{lem:optimgammamin} using $\gammamin \leq (1+L_1)^{-1}$. 
	The hook satisfies $g(y_{k+1}) = g(\Phi(\Retr_{y_k}(s_k))) \leq g(\Retr_{y_k}(s_k))$.
\end{proof}
\begin{lemma} \label{lem:optimlemma3}
	With Assumption~\ref{assu:upstairs}, Algorithm~\ref{algo:TRbalanced} performs at most $\ceil{\log_{\gamma_c}(\frac{\gammamin}{\overline{\gamma}})}$ unsuccessful iterations consecutively, where $k$ is an unsuccessful iteration if $\rho_k < \eta$.
\end{lemma}
\begin{proof}
	See~\cite[Lem.~2.6]{curtis2018concise}.
	The claim holds because $\gamma_k$ remains in
    $[\gammamin, \overline{\gamma}]$ for all $k$ and because each unsuccessful step involves a reduction as $\gamma_{k+1} = \gamma_c \gamma_k$.
\end{proof}

We are now ready to prove Theorem~\ref{thm:optimtheorem1}:
\begin{proof}[Proof of Theorem~\ref{thm:optimtheorem1}]
	The first part is~\cite[Lem.~2.11, Thm.~2.12]{curtis2018concise}.
	The argument involves a standard telescoping sum resting on Lemma~\ref{lem:optimlemma2} to bound the number of successful iterations which do not meet the $\epsilon_1, \epsilon_2$ targets.
	Then, Lemma~\ref{lem:optimlemma3} is used to argue that each such successful iteration was preceded by at most $\ceil{\log_{\gamma_c}\!\left(\gammamin / \overline{\gamma} \right)}$ unsuccessful iterations (with the same gradient and Hessian).
	For the second part, we use that only a finite number of iterates fail to meet any given tolerances $\epsilon_1, \epsilon_2$.
	Thus, for arbitrary tolerances, eventually, all iterates comply.
	Hence, any accumulation point $y$ satisfies $\|\nabla g(y)\| = 0$ and $\lambdamin(\nabla^2(g \circ \Retr_y)(0)) \geq 0$.
	Conclude with Lemma~\ref{lem:optim2criticalcharact}.
\end{proof}

\subsection{Securing Assumption~\ref{assu:upstairs}}\label{apdx:lip_assump}
In this section, we show that Assumption~\ref{assu:upstairs} holds when $g$ is three times continuously-differentiable and the iterates of Algorithm~\ref{algo:TRbalanced} remains in a compact set. We guarantee the latter condition using rebalancing maps in Section~\ref{sec:rebalancing_maps}.
%
%
%
%
%
\begin{lemma} \label{lem:optimlipschitzlemma}
	Consider a retraction $\Retr$ on $\calM$, a compact subset $\calL \subseteq \calM$ and a continuous, nonnegative function $\deltamax \colon \calL \to \reals$.
	The set
	\begin{align}
		\calB & = \left\{ (y, u) \in \T\calM : y \in \calL \textrm{ and } \|u\| \leq \deltamax(y) \right\}
		\label{eq:calB}
	\end{align}
	is compact in the tangent bundle $\T\calM$ of $\calM$.
	Assume $g \colon \calM \to \reals$ is twice continuously differentiable. 
	For the pullback $\hat g_y = g \circ \Retr_y$
    there exists a constant $L_1$ such that
	\begin{align}
		\forall (y, u) \in \calB, && \left| g(\Retr_y(u)) - g(y) - \inner{u}{\nabla \hat g_y(0)} \right| & \leq \frac{L_1}{2} \|u\|^2,
		\label{eq:glip1}
	\end{align}
	and $\|\nabla^2 \hat g_y(0)\| \leq L_1$ for all $y \in \calL$. 
	If additionally $g$ is three times continuously differentiable, then there exists a constant $L_2$ such that
	\begin{align}
		\forall (y, u) \in \calB, && \left| g(\Retr_y(u)) - g(y) - \inner{u}{\nabla \hat g_y(0)} - \frac{1}{2} \inner{u}{\nabla^2 \hat g_y(0)[u]} \right| & \leq \frac{L_2}{6} \|u\|^3, \label{eq:glip2} \\
		&& \left\| \nabla \hat g_y(u) - \nabla \hat g_y(0) - \nabla^2 \hat g_y(0)[u] \right\| & \leq \frac{L_2}{2} \|u\|^2. \label{eq:glip2bis}
	\end{align}
\end{lemma}
\begin{proof}
	The compactness of $\calB$ follows from compactness of $\calL$ and continuity of $\deltamax$. 
	The inequalities hold because:
	\begin{enumerate}
		\item If $g$ is $k$ times continuously differentiable, so is $g \circ \Retr_y$ as a function on $\T_y\calM$.
		\item Thus, derivatives of $g \circ \Retr_y$ of order up to $k$ are bounded inside the disk of radius $\deltamax(y)$,
		\item And those bounds are continuous in $y$, which ranges over a compact set $\calL$.
		\item We deduce that there exists $L_{k-1}$ such that the $k$th derivative of $g \circ \Retr_y$ at $u$ is bounded by $L_{k-1}$ for all $(x, u) \in \calB$.
	\end{enumerate}
	From here, we bound residuals of truncated Taylor expansions via the fundamental theorem of calculus on each individual tangent space. 
\end{proof}
\begin{corollary} \label{cor:optimassumptionlipschitzhelper}
	Assumption~\ref{assu:upstairs} holds if the iterates of Algorithm~\ref{algo:TRbalanced} remain in a compact set and $g$ is three times continuously differentiable.
\end{corollary}
\begin{proof}
	With $\deltamax(y) = \overline{\gamma} \cdot \max\left( \|\nabla g(y)\|, |\lambdamin(\nabla^2 \hat g_y(0))| \right)$ (which is continuous), the claim follows from Lemma~\ref{lem:optimlipschitzlemma} because the trust-region radius $\delta_k$ at iterate $y_k$ is bounded by $\deltamax(y_k)$ so that $(y_k, s_k)$ is in $\calB$~\eqref{eq:calB}.
\end{proof}

\section{Lifts and rebalancing maps for tensor varieties}
\label{apdx:tensors}
In this appendix we elaborate on Example~\ref{ex:tensors}, by showing that two well-known tensor decompositions give smooth lifts of singular varieties, and that standard SVD-based normalization algorithms for the decompositions are  sectional rebalancing maps.


\subsection{Tucker decomposition}\label{ex:tucker}
Fix positive integers $n_1, \ldots, n_d$ and $r_1, \ldots, r_d$ such that $n_k \geq r_k$ for each $k=1, \ldots, d$.  
Let $\mathcal{M} = \mathbb{R}^{r_1 \times \cdots \times r_d} \times \mathbb{R}^{n_1 \times r_1} \times \cdots \times \mathbb{R}^{n_d \times r_d}$ 
with its usual Euclidean structure, and $\mathcal{E} = \mathbb{R}^{n_1 \times \cdots \times n_d}$.  
So $\mathcal{M}$ consists of tuples $\left(\mathcal{C}, M^{(1)}, \ldots, M^{(d)}\right)$ where $\mathcal{C} \in \mathbb{R}^{r_1 \times \cdots \times r_d}$, $M^{(k)} \in \mathbb{R}^{n_k \times r_k}$. We call $\mathcal{C}$ the \textup{core tensor} and $M^{(k)}$ the $k$-th \textup{factor matrix}. 
Of course, $\mathcal{M}$ is a complete manifold.  
Meanwhile, $\mathcal{E}$ consists of bigger-sized tensors $\mathcal{T} \in \mathbb{R}^{n_1 \times \cdots \times n_d}$.

Consider $\varphi : \mathcal{M} \rightarrow \mathcal{E}$ given by \textup{Tucker (or Multilinear) Tensor Decomposition} \cite{de2000multilinear}:
\begin{align*}
  &  \varphi(\mathcal{C}, M^{(1)}, \ldots, M^{(d)}) = \mathcal{C} \times_1 M^{(1)} \times_2 \cdots \times_d M^{(d)}, \\
 &  \textup{where}  \,\, \left( \mathcal{C} \times_1 M^{(1)} \times_2 \cdots \times_d M^{(d)} \right)_{i_1, \ldots, i_d} := 
\sum_{j_1 = 1}^{r_1} \cdots \sum_{j_d = 1}^{r_d} \mathcal{C}_{j_1, \ldots, j_d} M^{(1)}_{i_1, j_1} \cdots M^{(d)}_{i_d, j_d},
\end{align*}
for $1 \leq i_k \leq n_k$.
The image of $\varphi$ is a closed real algebraic variety.  
One says that $\mathcal{X} = \varphi(\mathcal{M})$ consists of the tensors $\mathcal{T}$ in $\mathbb{R}^{n_1 \times \cdots \times n_d}$ of \textup{Tucker (or Multilinear) Rank} entrywise less than or equal to $(r_1, \ldots, r_d)$.  

 It turns out that the \textup{Higher-Order Singular Value Decomposition (HOSVD)} gives a sectional rebalancing map for $\varphi$.  Specifically, we let $\Phi(\mathcal{C}, M^{(1)}, \ldots, M^{(d)})$ be the HOSVD of $\mathcal{T} := \varphi(\mathcal{C}, M^{(1)}, \ldots, M^{(d)})$.
 This means that 
$$
\Phi(\mathcal{C}, M^{(1)}, \ldots, M^{(d)}) := (\mathcal{S}, U^{(1)}, \ldots, U^{(d)}) \in \mathcal{M},
$$
where 
$U^{(k)} \in \textup{St}(n_k,r_k)$ is a matrix of $r_k$ leading left singular vectors\footnote{If there are repeated singular values, we arbitrarily but deterministically choose a basis of singular vectors.} of the matrix flattening of $\mathcal{T} = \varphi(\mathcal{C}, M^{(1)}, \ldots, M^{(d)})$ along the $k$-th mode, and $\mathcal{S} \in \mathbb{R}^{r_1 \times \ldots \times r_k}$ is the induced core tensor.  

The HOSVD construction is made precise as follows.
Using a hat to denote an omitted term, the \textup{matrix flattening of $\mathcal{T}$ along the $k$-th mode} lies in $\mathbb{R}^{n_k \times (n_1 \ldots \widehat{n}_k \ldots n_d)}$ and is defined by
$$
\flatten_{(k)}(\mathcal{T})_{i_k, \left(i_1, \ldots, \widehat{i}_k, \ldots, i_d\right)} := \mathcal{T}_{i_1, \ldots, i_d}.
$$
Then, each factor matrix $U^{(k)} \in \mathbb{R}^{n_k \times r_k}$ has orthonormal columns consisting of $r_k$ leading left singular vectors of $\flatten_{(k)}(\mathcal{T})$.  
With $U^{(k)}$ thus defined,  $\mathcal{S} \in \mathbb{R}^{r_1 \times \cdots \times r_d}$ in HOSVD is given by
$$
\mathcal{S} := \mathcal{T} \times_1 \left( U^{(1)} \right)^{\!\!\top} \times_2 \cdots \times_d \left( U^{(d)} \right)^{\!\!\top} \in \mathbb{R}^{r_1 \times \cdots \times r_d}.
$$

Then, $\Phi$ satisfies $\varphi = \varphi \circ \Phi$, since
\begin{align*}
& \left(\mathcal{T} \times_1 \left( U^{(1)} \right)^{\!\!\top} \times_2 \cdots \times_d \left( U^{(d)} \right)^{\!\!\top}\right) \times_1 U^{(1)}  \times_2 \cdots \times_d U^{(d)} \\
& = \mathcal{T} \, \times_1 \, U^{(1)}\!\left( U^{(1)} \right)^{\!\!\top} \, \times_2 \, \cdots \, \times_d \, U^{(d)}\! \left( U^{(d)} \right)^{\!\!\top} 
= \mathcal{T}.
\end{align*}
Here the second equality is because $U^{(k)}\!\left( U^{(k)}\right)^{\!\!\top} \in \mathbb{R}^{n_k \times n_k}$ is the orthogonal projector in $\mathbb{R}^{n_k}$ onto the column space of $U^{(k)}$, and this column space equals the column space of $\flatten_{(k)}(\mathcal{T})$ by construction of $U^{(k)}$.
Also, $\Phi$ maps all nonempty fibers of $\varphi$ to singletons, since the HOSVD algorithm described above only depends on $\mathcal{T}$.
Finally, the boundedness property in Definition~\ref{def:goodPhi}(b) is ensured by
\begin{multline*}
\left\|\Phi\left(\mathcal{C}, M^{(1)}, \ldots, M^{(d)} \right)\right\| = \left\| \left(\mathcal{S}, U^{(1)}, \ldots, U^{(d)} \right) \right\| = \sqrt{\|\mathcal{S}\|^2 + \sum_{k=1}^d \|U^{(k)}\|^2} 
\\ = \sqrt{\|\mathcal{S}\|^2 + \sum_{k=1}^d r_k}
= \sqrt{\|\mathcal{S} \times_1 U^{(1)} \times_1 \cdots \times_d U^{(d)} \|^2 + \sum_{k=1}^d r_k}
\\ = \sqrt{\| \varphi(\mathcal{C}, M^{(1)}, \ldots, M^{(d)}) \|^2 + \sum_{k=1}^d r_k }. 
\end{multline*}
Hence, HOSVD gives a sectional rebalancing map for Tucker Tensor Decomposition.

\subsection{Tensor train decomposition}\label{ex:train}
Fix positive integers $n_1, \ldots, n_d$ as well as $r_0, \ldots, r_d$ 
such that $r_0 = r_{d} = 1$ and $r_k \leq r_{k-1}n_k$ for each $k=1,\ldots,d-1$.  
Let $\mathcal{M} = \mathbb{R}^{r_0 \times n_1 \times r_1} \times \mathbb{R}^{r_1 \times n_2 \times r_2} \times \cdots \times \mathbb{R}^{r_{d-1} \times n_d \times r_{d}}$ with its usual Euclidean structure, and $\mathcal{E} = \mathbb{R}^{n_1 \times \cdots \times n_d}$.  
So $\mathcal{M}$ consists of tuples $(\mathcal{G}^{(1)}, \ldots, \mathcal{G}^{(d)})$ where $\mathcal{G}^{(k)} \in \mathbb{R}^{r_{k-1} \times n_k \times r_k}$, and we call $\mathcal{G}^{(k)}$ the $k$-th three-way \textup{core tensor}. 
Of course, $\mathcal{M}$ is a complete manifold. 
Meanwhile, $\mathcal{E}$ consists of $d$-way tensors $\mathcal{T} \in \mathbb{R}^{n_1 \times \cdots \times n_d}$.

Consider $\varphi : \mathcal{M} \rightarrow \mathcal{E}$ given by the \textup{Tensor Train Decomposition (TT-D)} \cite{oseledets2011tensor}:
$$
\varphi(\mathcal{G}^{(1)}, \mathcal{G}^{(2)}, \ldots, \mathcal{G}^{(d)})_{i_1, i_2, \ldots, i_d} := \sum_{j_0=1}^{r_0} \sum_{j_1=1}^{r_1} \ldots \sum_{j_d=1}^{r_d} \mathcal{G}^{(1)}_{j_0, i_1, j_1} \mathcal{G}^{(2)}_{j_1, i_2, j_2} \cdots \mathcal{G}^{(d)}_{j_{d-1}, i_d, j_d},
$$
for $1 \leq i_k \leq n_k$.  
An equivalent way of writing TT-D is by using products of matrices:
$$
\varphi(\mathcal{G}^{(1)}, \mathcal{G}^{(2)}, \ldots, \mathcal{G}^{(d)})_{i_1, i_2, \ldots, i_d} = \mathcal{G}^{(1)}(i_1) \, \mathcal{G}^{(2)}(i_2) \, \cdots \, \mathcal{G}^{(d)}(i_d). 
$$
Here $\mathcal{G}^{(k)}(i_k) \in \mathbb{R}^{r_{k-1} \times r_k}$ is the matrix slice of $\mathcal{G}^{(k)}$ with second index fixed to be $i_k$, and the product gives a $1 \times 1$ scalar (since for each $i_1$ and $i_d$, the slices $\mathcal{G}^{(1)}(i_1)$ and $\mathcal{G}^{(d)}(i_d)$ are of size $1 \times r_1$ and $r_{d-1} \times 1$ respectively, because $r_0 = r_d = 1$).  

The image of $\varphi$ is a closed real algebraic variety.  
One says that $\mathcal{X} = \varphi(\mathcal{M})$ consists of the tensors $\mathcal{T}$ in $\mathbb{R}^{n_1 \times \cdots \times n_d}$ with \textup{Tensor Train Rank (TT-rank)} entrywise less than or equal to $(r_1, \ldots, r_d)$.

It turns out that the standard
\textup{Tensor Train Singular Value Decomposition (TT-SVD)} gives a sectional rebalancing map for $\varphi$. 
Specifically, we let $\Phi(\mathcal{G}^{(1)}, \ldots, \mathcal{G}^{(d)})$ be the TT-SVD of $\mathcal{T} := \varphi(\mathcal{G}^{(1)}, \ldots, \mathcal{G}^{(d)})$, 
$$
\Phi(\mathcal{G}^{(1)}, \ldots, \mathcal{G}^{(d)}) := (\mathcal{U}^{(1)}, \ldots, \mathcal{U}^{(d)}) \in \mathcal{M}, 
$$
where $\mathcal{U}^{(k)} \in \mathbb{R}^{r_{k-1} \times n_k \times r_k}$ are generated according to the following sequence of matrix SVDs (and we use the same conventions for matrix SVD as in Example \ref{ex:tucker}). 
\begin{itemize}
\item First update $\mathcal{T} \leftarrow \operatorname{reshape}(\mathcal{T}, [n_1, n_2 \cdots n_d])$,
\footnote{Here $\operatorname{reshape}$ is as in standard Matlab/NumPy code, and reorders the tensor entries into another array with the same total number of entries, using the lexicographic ordering.  Note that the second argument specifies the dimensions of the reshaped tensor.} 
so that $\mathcal{T} \in \mathbb{R}^{n_1 \times n_2 \cdots n_d}$.  
Define $\mathcal{U}^{(1)} \in \textup{St}(n_1, r_1)$ to be given by $r_1$ leading left singular vectors of $\mathcal{T}$.  
Then update $\mathcal{T} \leftarrow  \left(\mathcal{U}^{(1)}\right)^{\!\top} \! \mathcal{T}$, so that $\mathcal{T} \in \mathbb{R}^{r_1 \times n_2 \cdots n_d}$.  
Then identify $\mathcal{U}^{(1)}$ with a point in $\mathbb{R}^{r_0 \times n_1 \times r_1}$ (as $r_0=1$).

\item Next update $\mathcal{T} \leftarrow \operatorname{reshape}(\mathcal{T}, [r_1n_2, n_3 \cdots n_d])$, so that $\mathcal{T} \in \mathbb{R}^{r_1n_2 \times n_3 \cdots n_d}$.  
Define $\mathcal{U}^{(2)} \in \textup{St}(r_1n_2, r_2)$ to be given by $r_2$ leading left singular vectors of $\mathcal{T}$.  
Then update $\mathcal{T} \leftarrow \left( \mathcal{U}^{(2)} \right)^{\!\top}\! \mathcal{T}$, so that $\mathcal{T} \in \mathbb{R}^{r_2 \times n_3 \cdots n_d}$.  
Then update $\mathcal{U}^{(2)} \leftarrow \operatorname{reshape}(\mathcal{U}^{(2)}, [r_1, n_2, r_2])$, so that $\mathcal{U}^{(2)} \in \mathbb{R}^{r_1 \times n_2 \times r_2}$.

\item Continue updating $\mathcal{T}$ and generating $\mathcal{U}^{(k)}$ for $k = 3, \ldots, d-1$ as in the previous bullet.

\item In the end, $\mathcal{T} \in \mathbb{R}^{r_{d-1} \times n_d}$.  
Let $\mathcal{U}^{(d)} = \mathcal{T}$ and identify this with a point in $\mathbb{R}^{r_{d-1}\times n_d \times r_d}$ (as $r_d = 1$).
\end{itemize}

With TT-SVD thus defined, we have $\varphi \circ \Phi = \varphi$.  
This is because for $\mathcal{T} = \varphi(\mathcal{G}^{(1)}, \ldots, \mathcal{G}^{(d)})$, the matrix SVD in the $k$-th step of TT-SVD has exact rank $\leq r_k$.  Also,
$\Phi$ maps the nonempty fibers of $\varphi$ to singletons, since TT-SVD depends only on $\mathcal{T}$.  
Finally, the boundedness property in Definition~\ref{def:goodPhi}(b) is ensured \nolinebreak by
\begin{align*}
    \left\| \Phi\left(\mathcal{G}^{(1)}, \ldots, \mathcal{G}^{(d)}\right) \right\|  
    = \left\| \left( \mathcal{U}^{(1)}, \ldots, \mathcal{U}^{(d)} \right)  \right\| 
    &= \sqrt{  \sum_{k=1}^{d-1} \left\| \mathcal{U}^{(k)} \right\|^2 \, + \, \left\| \mathcal{U}^{(d)} \right\|^2} 
    \\ &= \sqrt{\sum_{k=1}^{d-1} r_k \, + \, \left\|\varphi(\mathcal{G}^{(1)}, \ldots, \mathcal{G}^{(d)})\right\|^2}. 
\end{align*}
Here, the third equality follows from the fact that $\operatorname{reshape}(\mathcal{U}^{(k)},[r_{k-1}n_k, r_k])$ has orthonormal columns for $k = 1, \ldots, d-1$,
and that
  the variable $\mathcal{T}$ updated within TT-SVD has its norm preserved throughout the process, so $\|\mathcal{U}^{(d)}\| = \|\varphi(\mathcal{G}^{(1)}, \ldots, \mathcal{G}^{(d)}) \| $.
Thus, TT-SVD gives a sectional rebalancing map for Tensor Train Decomposition.
\section{Sharpness of Theorem~\ref{thm:stable_2implies1_LR}}\label{apdx:stable_LR_sharp}
We show through a simple example
that the bounds in Theorem~\ref{thm:stable_2implies1_LR} are sharp up to constants.
Let $f(X) = \frac{1}{2}\|X\|^2$.
Note that $\nabla f(X) = X$ so $L_f = 1$.
Let $X = \tt{diag}(\epsilon^{2/3},0,\ldots,0)\in\RR^{m\times n}$ for $\epsilon>0$.
We view $X$ as a point of $\Rmnlr$ with $r \geq 1$ arbitrary.
Let $L\in\RR^{m\times r}$ and $R\in\RR^{n\times r}$ be defined as $\tt{diag}(\epsilon^{1/3},0,\ldots,0)$ in their respective spaces.
We clearly have $L^\top L = R^\top R$ and $LR^\top = X$.
We can compute
\begin{align*}
    \|\nabla g(L,R)\| = \sqrt{\|L^\top\nabla f(X)\|^2 + \|\nabla f(X)R\|^2} = \sqrt{2}\epsilon.
\end{align*}
Also, for any $\dot L,\dot R$, 
\begin{align*}
    \langle \nabla^2g(L,R)[\dot L,\dot R],[\dot L,\dot R] \rangle & = \langle\nabla^2f(X)[L\dot R^\top\! + \dot LR^\top],L\dot R^\top\! + \dot LR^\top \rangle + 2\langle\nabla f(X),\dot L\dot R^\top \rangle\\
    & \geq \|L\dot R^\top\! + \dot LR^\top \|^2 - 2\|X\|\cdot\|\dot L\dot R^\top \| \\
    & \geq - 2\epsilon^{2/3}\|\dot L\|\|\dot R\|
      \geq -\epsilon^{2/3}(\|\dot L\|^2 + \|\dot R\|^2).
\end{align*}
where the last inequality follows from AM-GM.
Thus, $(L, R)$ is an $(\epsilon_1, \epsilon_2)$-approximate 2-critical point for $g(L, R) = f(LR^\top)$ with $\epsilon_1 = \sqrt{2}\epsilon$ and $\epsilon_2 = \epsilon^{2/3}$.
We also have $\|\nabla f(X)\| = \|\nabla f(X)\|_{\mathrm{op}} = \epsilon^{2/3}$. 

Let us assess this example against Theorem~\ref{thm:stable_2implies1_LR}.
If $r > 1$, then $\sigma_r(X) = 0$.
In this case, we see that inequalities~\eqref{eq:opnormnablafXbound}-\eqref{eq:normProjTXnegnablabound} hold with equality.
On the other hand, if $r = 1$ then $\sigma_r(X) = \epsilon^{2/3}$.
Inequality~\eqref{eq:opnormnablafXbound} states
\begin{align*}
    \epsilon^{2/3} = \|\nabla f(X)\|_{\mathrm{op}} \leq \epsilon_2 + 2L_f\sigma_r(X) = \epsilon^{2/3} + 2 \epsilon^{2/3} = 3 \epsilon^{2/3}.
\end{align*}
Thus, that inequality is only off by a factor 3.
Inequality~\eqref{eq:normProjTXnegnablabound} states
\begin{align*}
    \epsilon^{2/3} = \|\nabla f(X)\| &= \|\Proj_{\T_X\Rmnlr}(-\nabla f(X))\|\\ &\leq \min\left\{\sqrt{\frac{2}{\sigma_r(X)}} \epsilon_1, \sqrt{\rank(\nabla f(X))}(\epsilon_2 + 2L_f\sigma_r(X))\right\}\\ &=\min\{2\epsilon^{2/3}, 3\epsilon^{2/3}\} = 2 \epsilon^{2/3}.
\end{align*}
Thus, that inequality is only off by a factor 2.

\section{Clarke regularity is preserved by local diffeomorphisms}
\label{apdx:clarke_regularity_loc_diffeo}
We first observe that diffeomorphisms induce invertible linear maps on tangent cones:
\begin{lemma} \label{lem:map_of_tangent_cones_and_duals}
    Suppose $h:V\to U$ is a $\mc C^1$ diffeomorphism defined on open sets $V\subseteq\calE'$ and $U\subseteq\calE$, where $\calE,\ \calE'$ are linear spaces. Suppose $Y\subseteq V$ is any set, $X=h(Y)\subseteq U$ is its image, $y\in Y$ is arbitrary and $x=h(y)\in X$. Then $\D h(x): \T_yY\to \T_xX$ is an invertible linear map between the two tangent cones.
\end{lemma}
\begin{proof}
    $\D h(y)$ is a linear map, and it maps $\T_yY$ to $T_xX$: If $v\in\T_yY$ then by Definition~\ref{def:tangentcone} we can write $v=\lim_k\frac{y_k-y}{\tau_k}\in \T_yY$ where $(y_k)_{k\geq 1}$ is a sequence converging to $y$ and $\tau_k\to0$. Because $h$ is $\mc C^1$, we have $\D h(y)[v] = \lim_k\frac{h(x_k)-h(x)}{\tau_k}\in T_xX$. Since $h$ is a $\mc C^1$ diffeomorphism, it has a $\mc C^1$ inverse $h^{-1}:X\to Y$ sending $h^{-1}(x)=y$, so its differential maps $\D h(x)^{-1}:T_xX\to T_yY$. The two maps of tangent cones are inverses of each other.
\end{proof}

\begin{proposition}\label{prop:sets_with_nbhds_diff_to_Clarke}
Suppose $\calX\subseteq\calE$ and that $x\in \calX$ has a neighborhood $U$ in $\calX$ that is $\mc C^1$-diffeomorphic to a Clarke regular set $Y$, meaning there is a $\mc C^1$ diffeomorphism $h$ defined on a neighborhood of $U$ in $\calE$ that maps $U$ to $Y$. Then $\calX$ is Clarke regular at $x$. In particular, $x$ is not apocalyptic in $\calX$.
\end{proposition}
\begin{proof}
Let $(x_k)_{k\geq 1}\subseteq\calX$ is a sequence converging to $x$. After dropping finitely many terms in the sequence we may assume $x_k\in U$ for all $k$. Let $y=h^{-1}(x)\in Y$ and $y_k=h^{-1}(x_k)$. Since $h^{-1}$ is continuous, we have $y_k\to y$.

By Lemma~\ref{lem:map_of_tangent_cones_and_duals}, $\D h(y)\colon\T_yY\to\T_x\calX$ is surjective, hence for any $u\in \T_x\calX$, we can find $v\in\T_yY$ satisfying $u=\D h(y)[v]$. Because $Y$ is Clarke regular, we have $\T_yY\subseteq \liminf_k\T_{y_k}Y$. We can therefore write $v=\lim_kv_k$ where $v_k\in \T_{y_k}Y$. Because $y\mapsto \D h(y)$ is continuous, we have 
\begin{align}
    \|\D h(y_k)[v_k] - \D h(y)[v]\| &\leq \|(\D h(y_k) - \D h(y))[v_k]\| + \|\D h(y)[v_k-v]\|\\&\leq \|\D h(y_k)-\D h(y)\|\cdot\|v_k\| + \|\D h(y)\|\cdot\|v_k-v\|,
\end{align} 
which vanishes as $k\to\infty$, hence $\lim_k\D h(y_k)[v_k]=\D h(y)[v]=u$. Therefore, if we set $u_k=\D h(y_k)[v_k]\in\T_{x_k}\calX$ we have $u=\lim_ku_k$. This shows $\T_x\calX\subseteq\liminf_k\T_{x_k}\calX$, hence $\calX$ is Clarke regular at $x$.
\end{proof}

\end{document}